\documentclass[11pt]{article}

\usepackage{epsfig}
\usepackage{amssymb, latexsym}
\usepackage{amscd}
\usepackage[all, cmtip]{xy}
\usepackage{epsf}
\usepackage[mathscr]{eucal}
\usepackage[normalem]{ulem}

\usepackage{tikz-cd}
\usetikzlibrary{matrix,arrows,decorations.pathmorphing}
\tikzcdset{every label/.append style = {font = \footnotesize}}

\usepackage{amsmath,amsfonts,amscd,amssymb,amsthm}

\usepackage[titletoc]{appendix}
\usepackage[sc]{titlesec}

\long\def\comment#1\endcomment{}

\theoremstyle{plain}

\newtheorem{theorem}{\sc Theorem}[section]
\newtheorem{lemma}[theorem]{\sc Lemma}

\newtheorem{prop}[theorem]{\sc Proposition}
\newtheorem{coroll}[theorem]{\sc Corollary}

\comment

\endcomment

\theoremstyle{plain}

\newtheorem{defn}[theorem]{\sc Definition}

\theoremstyle{exercise}
\newtheorem{remark}[theorem]{\sc Remark}

\makeatletter \@addtoreset{equation}{section} \makeatother

\def\eqref#1{\thetag{\ref{#1}}}

\let\latexref=\ref
\def\ref#1{{\normalfont{\latexref{#1}}}}

\newcommand{\ldot}{{\:\raisebox{2,3pt}{\text{\circle*{1.5}}}}}
\def\dlim_#1{{\displaystyle\lim_{#1}}^\hdot}

\newcommand{\cchar}{\operatorname{\sf char}}

\newcommand{\End}{\operatorname{End}}

\newcommand{\Ker}{\operatorname{{\rm Ker}}}

\newcommand{\id}{\operatorname{\rm id}}

\newcommand{\Mod}{\mathrm{Mod}}

\renewcommand{\Bar}{\mathrm{Bar}}
\newcommand{\Cobar}{\mathrm{Cobar}}

\newcommand{\Hom}{\mathrm{Hom}}

\newcommand{\Ind}{\mathrm{Ind}}
\newcommand{\Coind}{\mathrm{Coind}}

\newcommand{\RHom}{\mathrm{RHom}}

\newcommand{\op}{\mathrm{op}}

\newcommand{\colim}{\mathrm{colim}}

\newcommand{\pr}{\mathrm{pr}}

\newcommand{\Sets}{\mathscr{S}ets}

\newcommand{\Vect}{\mathscr{V}ect}

\newcommand{\Coalg}{{\mathscr{C}oalg}}

\newcommand{\Alg}{{\mathscr{A}lg}}

\newcommand{\CE}{\mathrm{CE}}

\renewcommand{\k}{\Bbbk}

\renewcommand{\cchar}{\mathrm{char}\ }

\newcommand{\Res}{\mathrm{Res}}

\newcommand{\ev}{\mathrm{ev}}

\newcommand{\Def}{\mathrm{Def}}

\newcommand{\ad}{\mathrm{ad}}

\newcommand{\Op}{{\mathscr{O}p}}

\newcommand{\Tw}{\mathrm{Tw}}
\newcommand{\sgn}{\mathit{sgn}}

\newcommand{\ash}{\textup{!`}}
\newcommand{\Conv}{\mathrm{Conv}}

\newcommand{\MC}{\mathrm{MC}}
\newcommand{\WBimod}{{\mathscr{WB}}}

\newcommand{\pronilp}{\mathrm{pronilp}}
\newcommand{\M}{\mathcal{M}}

\newcommand{\lev}{\mathit{lev}}

\newcommand{\rep}{\mathrm{rep}}

\title{\sc{On a higher structure on the operadic deformation complexes $\Def(e_n\to\mathcal{P})$}}

\author{\sc{Boris Shoikhet}}
\date{}

\begin{document}\maketitle

{\footnotesize
\begin{center}{\parbox{4,5in}{{\sc Abstract.}
In this paper, we prove that there is a canonical homotopy $(n+1)$-algebra structure on the shifted operadic deformation complex $\Def(e_n\to\mathcal{P})[-n]$ for any operad $\mathcal{P}$ and a map of operads $f\colon e_n\to\mathcal{P}$. This result generalizes the result of [T2], where the case $\mathcal{P}=\End_\Op(X)$ was considered. Another more computational proof of the same statement was recently sketched in [CW]. 

Our method combines the one of [T2] with the categorical algebra on the category of symmetric sequences, introduced in [R] and further developed in [KM] and [Fr1]. 
We define suitable deformation functors on $n$-coalgebras, which are considered as the ``non-commutative'' base of deformation, prove their representability, and translate properties of the functors to the corresponding properties of the representing objects. A new point, which makes the method more powerful, is to consider the argument of  our deformation theory as an object of the category of symmetric sequences of dg vector spaces, not as just a single dg vector space.

}}
\end{center}
}

\section{\sc Introduction}
\subsection{\sc }\label{section00}
In the beautiful paper [T2], Dima Tamarkin proved that, for an algebra $X$ over the operad $\mathsf{e}_n$, $n\ge 2$, the deformation complex $\Def(X)[-n]$ admits a natural structure of homotopy $(n+1)$-algebra. Here $\mathsf{e}_n=H_\ldot(E_n,\k)$ is the homology operad of the $n$-dimensional little discs operad $E_n$, and $\cchar\k=0$. In this paper, we consider $\mathsf{e}_n$ as a {\it non-unitary} operad, in the sense of [F2], that is, $\mathsf{e}_n(0)=0$.

If one considered the case $n=1$ and took the operad $\mathsf{Assoc}$ for $\mathsf{e}_1$, the claim would be the famous Deligne conjecture for Hochschild cochains; all known proofs of it use transcendental methods such as Drinfeld associators.\footnote{See Remark \ref{remintro} below.} The result of [T2] is proven purely algebraically, without any transcendental methods. Moreover if one took the Poisson operad for $\mathsf{e}_1$, the method of [T2] would work as well. 

Notice that the deformation complex $\Def_\mathcal{O}(X)$ of an algebra $X$ over a Koszul operad $\mathcal{O}$ is the same that the deformation complex of the corresponding map of operads
$\Def(\mathcal{O}\to \End_\Op(X))$, where $\End_\Op(X)$ is the endomorphism operad of $X$, $\End_\Op(X)(n)=\Hom_\k(X^{\otimes n},X)$. 

In this paper, we provide a proof of similar statement for the case of the general deformation complex $\Def(\mathsf{e}_n\xrightarrow{f}\mathcal{P})$, where $\mathcal{P}$ is an operad and $f$ is a map of operads. 

Throughout the paper, $\k$ denotes a field of characteristic 0. Denote by $\Vect(\k)$ the $\k$-linear abelian category of (unbounded) complexes of $\k$-vector spaces. 

We prove here the following statement:

\begin{theorem}\label{theorem1}
For any operad $\mathcal{P}$ in $Vect(\k)$, and a morphism of operads $t\colon \mathsf{e}_n\to\mathcal{P}$, the shifted deformation complex 
\begin{equation}
\Def(\mathsf{e}_n\xrightarrow{t}\mathcal{P})[-n]
\end{equation}
admits a natural structure of a homotopy $(n+1)$-algebra. Its underlying homotopy $\mathsf{Lie}\{-n\}$-structure is strict and is given by the operadic convolution Lie algebra.\footnote{For the case $\mathcal{P}=\End_\Op(X)$, this bracket can be thought of as the Gerstenhaber-like bracket.} 
\end{theorem}

For the case of deformation complex $\Def(\mathsf{e}_n\xrightarrow{\id}\mathsf{e}_n)[-n]$, the corresponding Lie bracket of degree $-n$ is homotopically trivial, and the complex $\Def(\mathsf{e}_n\xrightarrow{\id}\mathsf{e}_n)[-n]$ becomes a homotopy $(n+2)$-algebra. We are going to discuss it elsewhere.

\begin{remark}\label{remintro}{\rm
Strictly speaking, the Deligne conjecture is the statement that there is an action of the chain operad $C(E_2,\k)$ on the Hochschild cochains. For this statement, several proofs which work over $\mathbb{Q}$ without any transcendental methods are known, see [MS1,2], [BB], [BF], and [B1,2]+[T4]. On the other hand, a homotopy 2-algebra is an algebra over the operad $\mathsf{hoe}_2$ which is the Koszul resolution of the operad $\mathsf{e}_2$. Any known construction of quasi-isomorphism of operads $\mathsf{hoe}_2\to C(E_2,\k)$ uses transcendental methods.
}
\end{remark}

\comment
\subsection{\sc }
There is a re-interpretation of Theorem \ref{theorem1} via the derived Hom in the category of {\it weak bimodules} over an operad.

A map of operads $f\colon \mathcal{O}\to\mathcal{P}$ makes $\mathcal{P}$ a {\it weak bimodule over $\mathcal{O}$}.
The deformation complex complex $\Def(\mathcal{O}\xrightarrow{f}\mathcal{P})$ is identified with the derived Hom in the category $\WBimod(\mathcal{O})$ of weak $\mathcal{O}$-bimodules:
\begin{equation}\label{formulawb}
\Def(\mathcal{O}\xrightarrow{f}\mathcal{P})[1]=\RHom_{\WBimod(\mathcal{O})}(\mathcal{O},f^*\mathcal{P})
\end{equation}
For completeness, we provide a proof \eqref{formulawb} in Appendix.

Then we get
\begin{theorem}\label{theorem2}
Let $f\colon \mathsf{e}_n\to\mathcal{P}$ be a map of operads. Then $\RHom_{\WBimod(\mathsf{e}_n)}(\mathsf{e}_n,f^*\mathcal{P})[-n-1]$ has a natural structure of a homotopy $(n+1)$-algebra.
\end{theorem}

Theorem 2 provides a linear counter-part of the topological conjecture in [HS] and [Tur], proven there for $n=1$.
\endcomment

\subsection{\sc }
When the paper had almost been completed, the author found it out that the recent paper [CW, Section 3] contains a sketch of another proof of Theorem \ref{theorem1}, based on different ideas. The proof relies on a nice construction loc. cit., Section 3.1, providing an operadic twisting interpretation of the Kontsevich-Soibelman operad [KS, Sect. 5], what makes it possible to define its counterpart for any Hopf cooperad. The authors re-interpret the construction in [T2] as a map of graded operads $\mathsf{hoe}_{n+1}\to \mathsf{Br}_{n+1}$ (in the notations of [CW]). After that, everything reduces to a lengthy computation of the compatibility of this map with the differentials, loc.cit. Section 3.2 and Appendix A. Unfotunately, this computation was only sketched.

An advantage of our approach, compared with loc.cit., is being more conceptual and categorical, and not relying on computations.
The author thinks that this paper, even though being just an account on another proof, has its own right for existence. 

\subsection{\sc }\label{section01}
Let us outline the methods we employ to prove Theorem \ref{theorem1}. 
We combine the methods of [T2] (see Section \ref{sectionintrotam} below) with the ``categorical algebra'' on the category of symmetric sequences, developed in [R], [KM], [St].

The operads are defined as monoids in the category of symmetric sequences, with respect to the composition product. The composition product admits an inner Hom which is right adjoint with respect to the left factor, denoted by $[-,-]$.
In particular, for any symmetric sequence $X$, the inner Hom $[X,X]$ is an operad; the operad $\End_\Op(V)$ for a vector space $V$ is recovered as $[V^{(0)},V^{(0)}]$ where the symmetric sequence $V^{(0)}$ is $V^{(0)}(0)=V$ and $V^{(0)}(n)=0$ for $n\ne 0$. 

One can talk on left modules over an operad $\mathcal{O}$ in the category of symmetric sequences; it is a symmetric sequence $X$ with a map $\mathcal{O}\circ X\to X$ which is associative and such that $\id\in \mathcal{O}(1)$ acts as identity on $X$. A conventional algebra $V$ over $\mathcal{O}$ is recovered as the case of the symmetric sequence $V^{(0)}$ defined just above. To give a left $\mathcal{O}$-module structure on a symmetric sequence $X$ is the same as to give an operad maps $\mathcal{O}\to [X,X]$.

One can as well talk on right $\mathcal{Q}$-modules, where $\mathcal{Q}$ is an operad. In contrast with the left $\mathcal{Q}$-modules, the right $\mathcal{Q}$-modules form a $\k$-linear abelian category. For two right $\mathcal{Q}$-modules $X,Y$, one can define the relative internal hom $[X,Y]_\mathcal{Q}$, and develop the corresponding ``categorical algebra''. In particular, $[X,X]_\mathcal{Q}$ becomes an operad. It is due to [R], with subsequent developent made in [KM], [F1], [St].

If a symmetric sequence $X$ is an $\mathcal{O}{-}\mathcal{Q}$-bimodule, one gets a map of operads 
\begin{equation}\label{intro1}
\mathcal{O}\to [X,X]_\mathcal{Q}
\end{equation}

Our first goal is to extend the result of Tamarkin to a map of operad $f\colon \mathsf{e}_n\to [X,X]$, where $X$ is a symmetric sequence (originally it was proven for the case when $X$ has only arity 0 non-zero component). It does not meet any trouble. We  mention that the bar-complex $\Bar_\mathcal{O}(X)$ is defined as a {\it symmetric sequence} with a component-wise differential.

As the next step, we take a $\mathsf{e}_n{-}\mathcal{Q}$-bimodule $X$, and extend the Tamarkin theory [T2] for the corresponding operad morphism $\mathsf{e}_n\to [X,X]_\mathcal{Q}$, see \eqref{intro1}.
To this end, we prove the ``$\mathcal{Q}$-relative'' version of the classical statement, expressing maps of dg coalgebras $C\to \Bar_\mathcal{O}(X)$ as the Maurer-Cartan elements in the convolution Lie algebra $\Hom(C/\k,X)[-1]$ (here $\mathcal{O}$ is a Koszul operad such that the cooperad $\mathcal{O}^\ash$ is biaugmented, and the $\mathcal{O}^\ash$-coalgebra $C$ is pro-conilpotent, as well as for the classical case\footnote{See Section \ref{subsectioncoalg} for the definitions of a bi-augmented cooperad and a pro-conilpotent coalgebra over it}), see Section \ref{subsubsectionlierel}.

Then we apply it to the case $X=\mathcal{P}$, considered as a right module over $\mathcal{Q}=\mathcal{P}$. A map of operads $f\colon \mathcal{O}\to\mathcal{P}$ makes $\mathcal{P}$ an $\mathcal{O}{-}\mathcal{P}$-bimodule. One has:
\begin{equation}
[\mathcal{P},\mathcal{P}]_\mathcal{P}=\mathcal{P}
\end{equation}
as an operad, and the map \eqref{intro1}
\begin{equation}
\mathcal{O}\to [\mathcal{P},\mathcal{P}]_\mathcal{P}=\mathcal{P}
\end{equation}
is equal to $f$.

Thus, any map of operads appears as a case of the map \eqref{intro1}, and the general statement follows from the case of operad maps \eqref{intro1}.

\subsection{\sc }\label{sectionintrotam}
Let us briefly outline the main ideas of [T2], referring the reader to the Introduction to loc.cit. for a more detailed overview.

First of all, the deformation complex $\Def_\mathcal{O}(X)=\Def_\mathcal{O}(X\xrightarrow{\id}X)[1]$, where the r.h.s. is the deformation complex of the identity morphism of an $\mathcal{O}$-algebra $X$. One can consider more general deformation complexes $\Def_\mathcal{O}(X\xrightarrow{f}Y)$ of a morphism of $\mathcal{O}$-algebras, what provides a relative version of the initial deformation complex. The convention we adapt here is that all deformation complexes we deal with are considered with the grading making them a dg Lie algebra, that is, the underlying Lie bracket preserves the grading.

There is the following conceptual way to think on the deformation complex $\Def_\mathcal{O}(X\xrightarrow{f}Y)$. 
One associates with a morphism $f\colon X\to Y$ of algebras over a Koszul operad $\mathcal{O}$ a functor $F_{X,Y}^f$ on coaugmented cocommutative coalgebras with values in sets:
\begin{equation}\label{intro37}
F_{X,Y}^f(a)=\{\phi\in \Hom_{\Coalg(\mathcal{O}^\ash)}(a\otimes \Bar_\mathcal{O}(X),\Bar_{\mathcal{O}}(Y)), \phi\circ(\eta\otimes \id_{\Bar_{\mathcal{O}}})=\Bar(f)\}
\end{equation}
where $\eta\colon \k\to a$ is the coaugmentation map.
The bar-complexes are coalgebras over the (shifted) Koszul dual cooperad $\mathcal{O}^\ash$, and $\Hom$ is taken in the category of $\mathcal{O}^\ash$-coalgebras. The map $f$ defines a map $\Bar(f)\colon \Bar_\mathcal{O}(X)\to\Bar_\mathcal{O}(Y)$ of $\mathcal{O}^\ash$-coalgebras, and one considers the maps $\phi\colon a\otimes\Bar_\mathcal{O}(X)\to\Bar_\mathcal{O}(Y)$ of $\mathcal{O}^\ash$-coalgebras, equal to $\Bar(f)$ on the ``special (co)fiber''. Notice that, if one worked with the Schlessinger framework of deformation theory, one would consider the functor (on artinian coalgebras), defined as $a\mapsto \{t\in \Hom_{\Alg(\mathcal{O})}(\Hom_\k(a,X),Y), t\circ \eta=f\}$.
The functor $F_{X,Y}^f$ should be considered as a derived version of the latter functor; working with the Deligne-Drinfeld framework of deformation theory (which describes the deformations via the Maurer-Cartan elements in a dg Lie algebra), one should replace the functors themselves with their ``derived versions''.

If one restricts ourselves with a full subcategory of the category of cocommutative coalgebras called pro-conilpotent (aka connected in [Q], aka pro-artinian in [T2]), and assumes that the cooperad $\mathcal{O}^\ash$ is biaugmented, the functor $F_{X,Y}^f$ becomes representable:
$$
F_{X,Y}^f(a)=\Hom_{\Coalg}(a,a_\rep(f))
$$
for some pro-conilpotent cocommutative coalgebra $a_\rep(f)$, where $\Coalg$ stands for the category of pro-conilpotent cocommutative coalgebras over $\k$.  Tamarkin shows that $$a_\rep(f)=C_\CE(\Def_\mathcal{O}(X\xrightarrow{f}Y),\k)$$
where $C_\CE(-)$ stands for the Chevalley-Eilenberg chain complex of a dg Lie algebra. It defines the dg Lie algebra $\Def_\mathcal{O}(X\xrightarrow{f}Y)$ up to a quasi-isomorphism. The Lie bracket is the algebra-convolution Lie bracket, see Section \ref{section022}.

Then the composition property of the functor $F_{X,Y}^f$, saying that, for two maps
$$
X\xrightarrow{f}Y\xrightarrow{g}Z
$$
one has a (functorial) map
\begin{equation}\label{intro47}
F_{Y,Z}^g(a^\prime)\times F_{X,Y}^f(a)\to F_{X,Z}^{gf}(a^\prime\otimes a)
\end{equation}
is translated to the corresponding property of the representing coalgebras:
\begin{equation}\label{intro49}
a_\rep(f)\otimes a_\rep(g)\to a_\rep(gf)
\end{equation}
It is associative, for a chain of 3 morphisms. It follows that, for $f=\id:X\to X$, one has a monoid structure
$$
a_\rep(X)\otimes a_\rep(X)\to a_\rep(X)
$$
where $a_\rep(X)=a_\rep(\id_X)$.

The conclusion is that $a_\rep(X)=C_\CE(\Def_\mathcal{O}(X\xrightarrow{\id}X))$ becomes a cocommutative dg bialgebra (whose underying cocommutative coalgebra is cofree pro-conilpotent).
Then the Milnor-Moore theorem applied to the underlying bialgebra $a_\rep(X)$ (with forgotten differential) implies that the convolution Lie bracket on $\Def_\mathcal{O}(X\xrightarrow{\id}X)$ vanishes, and the shifted complex $\Def_\mathcal{O}(X\xrightarrow{\id}X)[1]$ gets a Lie bracket.

Indeed, the Milnor-Moore theorem says that $a_\rep(X)$ is the pro-conilpotent universal enveloping (co)algebra $\mathcal{U}(\mathfrak{g})$, where $\mathfrak{g}$ is the graded space of primitive elements in the coalgebra $a_\rep(X)$. One easily identifies $\mathfrak{g}$ with $\Def_\mathcal{O}(X\xrightarrow{\id}X)[1]$. The quadratic component of the Chevalley-Eilenberg differential on $\mathfrak{g}$ vanishes by definition. Now the Leibniz rule for the dg bialgebra $C_\CE(\Def_\mathcal{O}(X\xrightarrow{\id}X))$ shows that the quadratic component of the Chevalley-Eilenberg differential vanishes on the entire complex, therefore, the convolution Lie bracket on $\Def_\mathcal{O}(X\xrightarrow{\id}X)$ vanishes. On the other hand, $\mathfrak{g}=\Def_\mathcal{O}(X\xrightarrow{\id}X)[1]$ gets a new Lie bracket, $[a,b]=a*b-(-1)^{|a||b|}b*a$, where $a,b$ are primitive, and $-*-$ is the product on $a_\rep(X)$ (given by the monoid structure on $a_\rep(\id_X)$, as above). This Lie bracket is clearly compatible with the differential on $\Def_\mathcal{O}(X\xrightarrow{\id}X)[1]$, because the product $-*-$ is. 

It is the Lie bracket of Gerstenhaber type\footnote{Here we mean that the Lie bracket is defined similarly with the Gerstenhaber bracket on Hochschild cochain complex}.

Note that for $f\ne\id$, the algebra convolution Lie bracket on $\Def_\mathcal{O}(X\xrightarrow{f}Y)$ is not 0. Therefore, the above construction can be thought of as a sort of  ``quasi-classical limit''.

Notice that what makes all these constructions possible, is the the fact that, for any cooperad $\mathcal{C}$ and a $\mathcal{C}$-coalgebra $B$, the tensor product $a\otimes B$ with a cocommutative coalgebra $a$ is again a $\mathcal{C}$-coalgebra. On the language of operads, it follows from existence of a map of cooperads $\mathsf{Comm}^*\otimes \mathcal{C}\to\mathcal{C}$ (which holds as $\mathsf{Comm}^*=\mathsf{Comm}$ is the unit for the product $\otimes$ in symmetric sequences, so $\mathsf{Comm}^*\otimes\mathcal{C}=\mathcal{C}$).

A new fundamental idea in [T2] was to take for $a$ higher structured algebraic objects than cocommutative coalgebras. That is, D.Tamarkin considers ``deformations with non-commutative base''.
Namely, he considers the case when $\mathcal{C}=\mathcal{O}^\ash=\mathsf{e}_n^*$, the dual cooperad to the operad $\mathsf{e}_n$. The operad $\mathsf{e}_n$ is a Hopf operad, that is, there is an operad map 
\begin{equation}\label{intro77}
 \mathsf{e}_n\to\mathsf{e}_n\otimes\mathsf{e}_n
\end{equation}
and thus, for the dual cooperad one has 
\begin{equation}\label{intro78}
\mathsf{e}_n^*\otimes\mathsf{e}_n^*\to\mathsf{e}_n^*
\end{equation}
(we call it Hopf cooperad). (In fact, the coproduct like \eqref{intro77} exists for any operad in graded vector spaces defined as the homology operad of a topological operad. Indeed, for the topological spaces there is the arity-wise diagonal map. Passing to homology, it indices a coproduct).

It follows from \eqref{intro77} that the tensor product of two $\mathsf{e}_n$-algebras is again a $\mathsf{e}_n$-algebra, and it follows from \eqref{intro78} that the tensor product of two $\mathsf{e}_n^*$-coalgebras is again a $\mathsf{e}_n^*$-coalgebra.

The  case $\mathcal{C}=\mathcal{O}^\ash=\mathsf{e}_n^*$ is corresponded to $\mathcal{O}=\mathsf{e}_n\{n\}$ (by $\mathcal{O}\{1\}$ is denoted the operadic shift, see Section \ref{sectionshift}).

It follows that for $\mathcal{O}=\mathsf{e}_n\{n\}$ one can upgrade the functor $F_{X,Y}^f$ to another functor $G_{X,Y}^f$ defined exactly as \eqref{intro37}, but with $a$ a $\mathsf{e}_n^*$-coalgebra. 
So $G_{X,Y}^f$ is a functor from the category $\Coalg_n$ of pro-conilpotent $n$-coalgebras to the category $\Sets$. 

The functor $G_{X,Y}^f$ is also representable, and the representating object is 
$$A_\rep(f)=\Bar_{\mathsf{e}_n\{n\}}(\Def(X\xrightarrow{f}Y)[1])$$
where $\Def(X\xrightarrow{f}Y)[1]$ is endowed with a $\mathsf{e}_n\{n\}$-algebra structure.  It has the property similar to \eqref{intro47}, translated to a map \eqref{intro49} of $n$-coalgebras. It implies that $A_\rep(X):=A_\rep(\id_X)$ is a monoid object in the category of $n$-coalgebras.
Note that $X$ may be a $\mathsf{e}_n\{k\}$-algebra for any $k$, as, by an operadic shift, a map of $\mathsf{e}_n\{k\}$-algebras can be made a map of $\mathsf{e}_n\{n\}$ algebras, as the construction requires. 

Moreover, there is an imbedding $i\colon \Coalg\to\Coalg_n$, and the restriction of the functor $G_{X,Y}^f$ along it is equal to $F_{X,Y}^f$. The functor $i$ admits right adjoint, what makes possible to link the representating objects $a_\rep(X)$ and $A_\rep(X)$, along with their monoid structures. 

These constructions provide us with some higher structure on the deformation complex $\Def(X\xrightarrow{\id}X)$, where $X$ is an $\mathsf{e}_n$-algebra. Having unwound this structure, D.Tamarkin derived a homotopy $(n+1)$-algebra structure on $\Def_{\mathsf{e}_n}(X)[-n]$.

To be more precise, to get a homotopy $\mathsf{e}_{n+1}$ algebra structure on $\Def_{\mathsf{e}_n}(X)[-n]$, one uses solely the deformation functor $G_{X,Y}^f$, corresponded to the deformation theory with $\mathsf{e}_n$ coalgebra base, and the monoid structure on the $\mathsf{e}_n\{n\}$ coalgebra $A_\rep(X)$, see Section \ref{section61}. The comparison of the two deformation theories, given by the functors $F_{X,Y}^f$ and $G_{X,Y}^f$, via the right adjoint $R$ to the functor $i\colon \Coalg\to\Coalg_n$, is employed to get an explicit description of the underlying homotopy Lie bracket on $\Def_{\mathsf{e}_n}(X)$ of the $\mathsf{e}_{n+1}$ algebra structure on $\Def_{\mathsf{e}_n}(X)[-n]$. Namely, this bracket is identified with the operadic convolution (strict) Lie bracket. The functor $i$ is (colax-)monoidal, therefore, its right adjoint is lax monoidal. As such, it sends monoids to monoids. The key point is to show that the right adjoint functor $R$ sends $A_\rep(X)$ to $a_\rep(X)$ {\it with their monoid structures}, see Section \ref{section62}.

\comment

\subsection{\sc}\label{section03}
Let us recall briefly the main idea of [T2].

Let $\mathcal{O}$ be a Koszul operad, $f\colon X\to Y$ a map of $\mathcal{O}$-algebras. Then one defines a functor from pro-conilpotent dg cocommutative coalgebras to sets
\begin{equation}\label{intro5}
a\mapsto \{\phi \in \Hom_{\Coalg(\mathcal{O}^\ash)}(a\otimes\Bar_\mathcal{O}(X),\Bar_\mathcal{O}(Y),\ \phi\circ \eta=\Bar(f)\}
\end{equation}
where $\eta\colon \k\to a$ is the coaugmendation map. Denote this functor by $F_{X,Y}^f(-)$.

This functor is always representable. Denote by $a_\rep(X\xrightarrow{f}Y)$ the pro-conilpotent cocommutative coalgebra, representing this functor. explicitly, $a_\rep(X\xrightarrow{f}Y)$ is equal to the Chevalley-Eilenberg chain complex 
$$
a_\rep(X\xrightarrow{f}Y)=C_\CE(\Def(X\xrightarrow{f}Y),\k)
$$
where $\Def(X\xrightarrow{f}Y)$ is the deformation complex, equipped with a dg Lie algebra structure. In fact, it may be viewed as a definition of the deformation complex. 

Let $X\xrightarrow{f}Y\xrightarrow{g}Z$ be two maps of $\mathcal{O}$-algebras. Then one sees immediately that there is a bi-functorial in $a,a^\prime$ map of sets
\begin{equation}\label{intro6}
F_{X,Y}^f(a)\times F_{Y,Z}^g(a^\prime)\to F_{X,Z}^{gf}(a\otimes a^\prime)
\end{equation}
It implies that for the representing coalgebras one has a product
$$
a_\rep(X\xrightarrow{f}Y)\otimes a_\rep(Y\xrightarrow{g}Z)\to a_\rep(X\xrightarrow{gf}Z)
$$
It is associative, for three maps $X\to Y\to Z\to W$.
In particular, take $X=Y=Z$, $f=g=\id$. Denote $a_\rep(X)=a_\rep(X\xrightarrow{\id}X)$. One gets a map of dg pro-conilpotent cocommutative coalgebras
\begin{equation}\label{intro7}
a_\rep(X)\otimes a_\rep(X)\to a_\rep(X)
\end{equation}

It is an associative product. It makes $a_\rep(X)$ a cocommutative dg bialgebra.

It already gives some ``higher structure'' on $\Def(X\xrightarrow{\id}X)$, but by so far not that interesting. To encompass the higher structure hidden in the product \eqref{intro7}, one applies the the Milnor-Moore theorem to the cocommutative bialgebra $a_\rep(X)$. 
The consequence is that the ``initial'' Lie bracket on $\Def(X\xrightarrow{\id}X)$ vanishes, and the shifted complex $\Def(X\xrightarrow{\id}X)[1]$ yields a Lie bracket. This bracket is the Gerstenhaber-like bracket. Note that for $f=\id$, the ``initial'' Lie bracket on $\Def(X\xrightarrow{f}Y)$ may be non-trivial. Thus, the whole procedure may seem as a kind of ``quasi-classical limit''.

The idea in [T2] is that, for special operads $\mathcal{O}$, ``the base'' of the deformation can be choosen as a more structured than a cocommutative coalgebra object.

\endcomment

\subsection{\sc Organisation of the paper}
The paper is organized as follows.

In Section 2 we recall some facts on operads which will be used in the paper. Nothing here is new. 

In Section 3 we recall the bar-cobar duality, for algebras $X$ over a Koszul operad $\mathcal{O}$. We extend the classical statement to the case when $X$ is a symmetric sequence, that is, is a left module over $\mathcal{O}$. The bar and cobar complexes become symmetric sequences with component-wise differentials, equivariant with respect to the action of the symmetric groups. In Section 3.5 we consider the situation when $X$ is a $\mathcal{O}{-}\mathcal{Q}$-bimodule, for an operad $\mathcal{Q}$. Then the bar-complex is at once a coalgebra over the cooperad $\mathcal{O}^\ash$ and a right module over the operad $\mathcal{Q}$. We prove a fragment of the correspong bar-cobar duality statement, for this setting. 

Section 4 contains the deformation theory with cocommutative cobase. We start with the case when the target operad $\mathcal{P}$ of the map $f:\mathsf{e}_n\to\mathcal{P}$, we consider the deformations of, is $[X,X]$, for a symmetric sequence $X$. Then we consider the case $\mathcal{P}=[X,X]_\mathcal{Q}$, for $X$ a $\mathsf{e}_n{-}\mathcal{Q}$-bimodule. As we have shown above, this case covers an arbitrary $\mathcal{P}$.

In Section 5, we consider the deformation theory with $\mathsf{e}_n$-coalgebra as the cobase. We extend the results of Section 4 to this context.

Finally, Section 6 contains a derivation of Theorem \ref{theorem1} from the previous results. This part is very similar to [T2, Section 5], though we tried to clarify some points.

\hspace{2mm}
\subsection*{\sc }

\subsubsection*{\sc Acknowledgements}
The author is thankful to Michael Batanin, Anton Gerasimov, and Dima Tamarkin, for valuable discussions.
The author is thankful to the two anonymous referees for their careful reading of the paper and many remarks and corrections, which have undoubtedly made this published version more readable. 

The work was partially supported by the FWO Research Project Nr. G060118N and 
by the Russian Academic Excellence Project `5-100'.

\section{\sc Preliminaries}
Throughout the paper, $\k$ is a field of characteristic 0. We denote by $\Vect(\k)$ the symmetric monoidal category of ($\mathbb{Z}$-graded) dg vector spaces over $\k$. We use the cohomological grading, so all differentials are of degree +1. For a (dg) vector space $V$ with an action of a group $G$, we denote by $V^G$ the subspace of invariants, and by $V_G$ the quotient-space of coinvariants. 
\subsection{\sc Operads}\label{section01}
\subsubsection{\sc Symmetric sequences}
A (symmetric) operad $\mathcal{P}$ in $\Vect(\k)$ is a collection of $\Sigma_n$-representations $\mathcal{P}(n)$, $n\ge 0$, with the operadic composition operation 
\begin{equation}
\Upsilon\colon\mathcal{P}(k)\otimes\mathcal{P}(n_1)\otimes\dots\otimes \mathcal{P}(n_k)\to\mathcal{P}(n_1+\dots+n_k)
\end{equation}
and an element $\id\in \mathcal{P}(1)$, such that the operadic composition is compatible with the actions of symmetric groups and is associative (in some natural sense), and $\id$ is a two-sided unit for the operadic composition. 

Alternatively, one can define an operad a monoid in the monoidal category of {\it symmetric sequences}, with respect to the composition product.

By definition, a symmetric sequence in a category $\mathscr{C}$ is a collection of objects
$\{\mathcal{P}(n)\in \mathscr{C}\}_{n\ge 0}$ with an action of symmetric group $\Sigma_n$ on $\mathcal{P}(n)$. A morphism $\phi\colon \mathcal{P}\to\mathcal{Q}$ of symmetric sequences is a collection of $\Sigma_n$-equivariant maps $\{\mathcal{P}(n)\to\mathcal{Q}(n)\}_{n\ge 0}$:
\begin{equation}
\Hom_\Sigma(\mathcal{P},\mathcal{Q})=\prod_{n\ge 0}\Hom_\k(\mathcal{P}(n),\mathcal{Q}(n))^{\Sigma_n}
\end{equation}
We denote the category of symmetric sequences in $\mathscr{C}$ with the morphisms $\Hom_\Sigma(-,-)$ by $\mathscr{C}_\Sigma$.

\begin{defn}\label{dgsym}{\rm
A {\it dg symmetric sequence} is a symmetric sequence $\mathcal{P}$ in $\Vect(\k)_\Sigma$ (that is, each $\mathcal{P}(n)$ is a complex of $\k$-vector spaces), such that the differential on $\mathcal{P}(n)$ commutes with the action of $\Sigma_n$ on it, for each $n=0,1,2,\dots$.
}
\end{defn}

It is a monoidal category, with the monoidal product $\mathcal{P}\circ \mathcal{Q}$ defined as
\begin{equation}\label{f1}
(\mathcal{P}\circ \mathcal{Q})(n)=\bigoplus_{k\ge 0}\mathcal{P}(k)\otimes_{\Sigma_k}\mathcal{Q}^{\boxtimes k}(n)
\end{equation}
where 
\begin{equation}
\mathcal{Q}^{\boxtimes k}(n)=\oplus_{n_1+\dots+n_k=n}\mathrm{Ind}_{\Sigma_{n_1}\times\dots\times \Sigma_{n_k}}^{\Sigma_n}(\mathcal{Q}(n_1)\otimes\dots\otimes \mathcal{Q}(n_k))
\end{equation}
This product on the category of symmetric sequences is called {\it the composition product}, it is associative:
\begin{equation}\label{asscirc}
(\mathcal{P}\circ \mathcal{Q})\circ \mathcal{R}=\mathcal{P}\circ (\mathcal{Q}\circ \mathcal{R})
\end{equation}
Its unit is the symmetric sequence $I$, defined as $I(1)=\k$, $I(n)=0$ for $n\ne 1$. 
\begin{remark}{\rm
Note that for a group $G$ and two $G$-modules $V$ and $W$, the tensor product $V\otimes_G W$ means the space of {\it coinvariants}
$(V\otimes_\k W)_G$. That is, in \eqref{f1} one takes the {\it coinvariants} with respect to the group $\Sigma_k$ action.
}
\end{remark}
One defines version(s) of \eqref{f1} for {\it invariants} instead of coinvariants:
\begin{equation}\label{f1bis2}
(\mathcal{P}\circ_1 \mathcal{Q})(n)=\bigoplus_{k\ge 0}\Bigl(\mathcal{P}(k)\otimes\mathcal{Q}^{\boxtimes k}(n)\Bigr)^{\Sigma_k}
\end{equation}
\begin{equation}\label{f1bis}
(\mathcal{P}\hat{\circ}_1 \mathcal{Q})(n)=\prod_{k\ge 0}\Bigl(\mathcal{P}(k)\otimes\mathcal{Q}^{\boxtimes k}(n)\Bigr)^{\Sigma_k}
\end{equation}
These products are monoidal products in $\Vect(\k)_\Sigma$, with the same unit $I$. 

One re-interprets the concept of operad saying that {\it an operad is a monoid in the category of symmetric sequences, with respect to the composition product}. That is, a symmetric sequence $\mathcal{P}$ is an operad, if there are maps
\begin{equation}
\Upsilon\colon \mathcal{P}\circ \mathcal{P}\to\mathcal{P},\ \ i\colon I\to\mathcal{P}
\end{equation}
satisfying the axioms of a monoid in a monoidal category. An operad $\mathcal{P}$ such that $\mathcal{P}(n)=0$ unless $n=1$, $\mathcal{P}(1)=A$ is the same as an associative algebra. 
More generally, consider the functor 
$$
i\colon \Vect(\k)\to\Vect(\k)_\Sigma, \ \ V\mapsto X_V=(0,\underset{n=1}{V},0,0,\dots)
$$
It is a strict monoidal functor 
$$
i\colon (\Vect(\k),\otimes)\to (\Vect(\k)_\Sigma,\circ)
$$

A {\it cooperad} in $\Vect(\k)$ is a comonoid in $\Vect(\k)_\Sigma$ with respect to the monoidal product $-\hat{\circ}_1-$. That is, a cooperad is a symmetric sequence $\mathcal{C}$ with maps
\begin{equation}\label{coop}
\Theta\colon \mathcal{C}\to\mathcal{C}\hat{\circ}_1\mathcal{C},\ \ \varepsilon\colon \mathcal{C}\to I
\end{equation}
satisfying the usual comonoid axioms. 

A cooperad is called {\it finite} if the structure map \eqref{coop} factors as
\begin{equation}\label{coopbis}
\mathcal{C}\to\mathcal{C}{\circ}_1\mathcal{C}\to \mathcal{C}\hat{\circ}_1\mathcal{C}
\end{equation}

The composition product $-\circ-$ admits an inner hom, defined as 
\begin{equation}
[\mathcal{P},\mathcal{Q}](n)=\Hom_{\Sigma}(\mathcal{P}^{\boxtimes n},\mathcal{Q})
\end{equation} 
It is right adjoint to $-\circ-$:
\begin{equation}\label{adjcomp}
\Hom_\Sigma(\mathcal{P}\circ\mathcal{Q},\mathcal{R})=\Hom_\Sigma(\mathcal{P},[\mathcal{Q},\mathcal{R}])
\end{equation}
For any symmetric sequence $X$, $[X,X]$ is a monoid with respect to $\circ$, thus an operad.

\comment
There is a 3-functorial map
\begin{equation}\label{adjcomp1}
\Hom_\Sigma(  \mathcal{P}   ,\mathcal{Q}\circ_1\mathcal{R})\to \Hom_{\Sigma}([\mathcal{R},\mathcal{P}],\mathcal{Q})
\end{equation}
which is not, in general, an isomorphism.
\endcomment

An important particular case is when $X(n)=0$ unless $n=0$, $X(0)=W$. Then  
$[X,X](n)=\Hom_\k(W^{\otimes n},W)$. It is an operad with the operadic composition defined by plugging the arguments. We denote this operad by $\End_\Op(X)$. 

For any group $G$ and a $G$-module $M$, there is a map from invariants to coinvariants:
\begin{equation}\label{can0}
can\colon M^G\to M_G
\end{equation}
In particular, for any two symmetric sequences $X,Y$, one has a canonical map
\begin{equation}\label{can}
can_{X,Y}\colon X\circ_1 Y\to X\circ Y
\end{equation}
If the group $G$ is finite, and the order $\sharp G$ is invertible in $\k$, the map \eqref{can0} is an isomorphism, with the inverse given by the norm:
$$
N\colon M_G\to M^G
$$
\begin{equation}
m\mapsto \frac1{\sharp G}\sum_{g\in G}g\circ m
\end{equation}
In this paper, $\cchar \k=0$, and the map \eqref{can} is always an isomorphism.

We often identify the products $-\circ-$ and $-\circ_1-$, assuming these isomorphisms.

\subsubsection{\sc }
Along with the composition monoidal product on the category $\Vect(\k)_\Sigma$, we will consider the {\it  level-wise monoidal product} $-\otimes_\lev -$, defined as:
\begin{equation}
(X\otimes_\lev Y)(n)=X(n)\otimes Y(n)
\end{equation}
with the diagonal action of the symmetric group $\Sigma_n$. It is called {\it the Hadamard product} and is denoted by $-\otimes_H-$ in [LV].

The unit for this product is the symmetric sequence $\mathsf{Comm}$, defined as
\begin{equation}
\mathsf{Comm}(n)=\k, n\ge 0
\end{equation}
This product admits the internal Hom defined as
\begin{equation}
\Hom_\lev(X,Y)=\{\Hom_\k(X(n),Y(n))\}_{n\ge 0}
\end{equation}
with the symmetric group acting as $(\sigma*f)(x)=\sigma(f(\sigma^{-1}x))$. It is right adjoint to $-\otimes_\lev-$:
\begin{equation}\label{adjlev1}
\Hom_\Sigma(X\otimes_\lev Y,Z)=\Hom_\Sigma(X,\Hom_\lev(Y,Z))
\end{equation}
In fact, one has a stronger adjunction:
\begin{equation}\label{adjlev2}
\Hom_\lev(X\otimes_\lev Y,Z)=\Hom_\lev(X,\Hom_\lev(Y,Z))
\end{equation}

\begin{lemma}
Let $X,Y,X_1,Y_1$ be symmetric sequences in $\Vect(\k)$. One has the following 4-functorial maps:
\begin{itemize}
\item[(i)]
\begin{equation}\label{strange1}
\eta_{X_1X_2Y_1Y_2}\colon (X_1\otimes_\lev X_2)\circ (Y_1\otimes_\lev Y_2)\to (X_1\circ Y_1)\otimes_\lev (X_2\circ Y_2)
\end{equation}
\item[(ii)]
\begin{equation}\label{strange2}
\mu_{X_1X_2Y_1Y_2}\colon (X_1\circ_1 Y_1)\otimes_\lev (X_2\circ_1 Y_2)\to (X_1\otimes_\lev X_2)\circ_1 (Y_1\otimes_\lev Y_2)
\end{equation}
\end{itemize}
\end{lemma}
\begin{proof}
We use the fact that for a group $G$ and two $G$-modules there are maps
\begin{equation}\label{invcoinv}
V^G\otimes W^G\to (V\otimes W)^G\text{  and  }(V\otimes W)_G\to V_G\otimes W_G
\end{equation}
making the invariants functor a lax-monoidal functor, and the coinvariants functor a colax-monoidal functor.

Denote
\begin{equation}
(\mathcal{P}\star \mathcal{Q})(n)=\bigoplus_{k\ge 0}\mathcal{P}(k)\otimes_\k\mathcal{Q}^{\boxtimes k}(n)
\end{equation}
and consider the r.h.s. as a $\Sigma_k$-module.

There are maps (the imbedding $i$ and the projection $p$):
\begin{equation}
i\colon (X_1\otimes_\lev X_2)\star (Y_1\otimes_\lev Y_2)\to (X_1\star Y_1)\otimes_\lev (X_2\star Y_2)
\end{equation}
\begin{equation}
p\colon (X_1\star Y_1)\otimes_\lev (X_2\star Y_2)\to (X_1\otimes_\lev X_2)\star(Y_1\otimes_\lev Y_2)
\end{equation}
Both are maps of $\Sigma_k$-modules. The statement follows from \eqref{invcoinv}.

\end{proof}
\begin{coroll}
The level-wise tensor product of two operads is an operad. The level-wise tensor product of two finite cooperads is a (finite) cooperad.
\end{coroll}
\qed

Another important fact is
\begin{prop}\label{lemmabm}
For any four symmetric sequences $X,Y,X_1,Y_1$ in $\Vect(\k)$, there are 4-functorial maps
\begin{equation}\label{strange3}
\Hom_\lev(X,Y)\circ \Hom_\lev(X_1,Y_1)\to \Hom_\lev(X\circ X_1,Y\circ Y)
\end{equation}
and
\begin{equation}\label{strange33}
\Hom_\lev(X\circ X_1,Y\circ Y)\to \Hom_\lev(X,Y)\circ \Hom_\lev(X_1,Y_1)
\end{equation}
\end{prop}

\comment
For completeness, we provide a proof in Appendix A.
\endcomment

\qed

\begin{coroll}\label{corollbm}
Let $\mathcal{P}$ be an operad and let $\mathcal{C}$ be a finite cooperad in $\Vect(\k)$, $\cchar\ \k=0$. Then 
$$
\Hom_\lev(\mathcal{C},\mathcal{P})
$$
is naturally an operad, and
$$
\Hom_\lev(\mathcal{P},\mathcal{C})
$$
is naturally a finite cooperad.
\end{coroll}

\subsubsection{\sc Left and right modules, bimodules}
A {\it left (resp., right) module} over an operad $\mathcal{P}$ is a symmetric sequence $M$ with a map $\mathcal{P}\circ M\to M$ (resp., $M\circ \mathcal{P}\to M$) satisfying the usual module axioms. 

Let us stress an essential difference between the left and the right modules over an operad $\mathcal{P}$: the right modules over $\mathcal{P}$ always form a $\k$-linear abelian category (in the dg situation, we upgrade it to the corresponding dg category over $\k$), whereas the category of left $\mathcal{P}$-modules is even non-addtive.

{\it An algebra $X$ over an operad $\mathcal{P}$} is a symmetric sequence $X\in\Vect(\k)_\Sigma$ equipped with a map of operads $\mathcal{P}\to [X,X]$. Alternatively, it is the same that a left $\mathcal{P}$-module structure on $X$:
\begin{equation}
m\colon \mathcal{P}\circ X\to X
\end{equation}
We see from the adjunction \eqref{adjcomp} that $X$ is an algebra over an operad $\mathcal{P}$ iff it is a left module over $\mathcal{P}$. 

When $X(n)=0$ unless $n\ne 0$, $X(0)=X_0\in \Vect(\k)$, we call $X_0$ a {\it conventional} algebra over the operad $\mathcal{P}$.

For an operad $\mathcal{P}$ and $A\in\Vect(\k)_\Sigma$ consider 
\begin{equation}
\mathcal{P}\langle A \rangle=\mathcal{P}\circ A
\end{equation}

It follows from the associativity \eqref{asscirc} that $\mathcal{P}\langle A\rangle$ is an algebra over $\mathcal{P}$. The functor $A\rightsquigarrow \mathcal{P}\langle A\rangle$ is left adjoint to the forgetful functor from algebras over $\mathcal{P}$ to $\Vect(\k)_\Sigma$.

Let $X$ be a symmetric sequence with a left action of the operad $\mathcal{P}$.
Then $X\circ \k^{(0)}=X\langle \k\rangle$ is an algebra over $\mathcal{P}$. One has explicitly:
\begin{equation}
X\circ \k^{(0)}=\bigoplus_{n\ge 0}X(n)_{\Sigma_n}
\end{equation}

Let $\mathcal{P}_1,\mathcal{P}_2$ be operads. A $(\mathcal{P}_1,\mathcal{P}_2)$-bimodule is a symmetric sequence $\mathcal{M}$, with a left module structure over $\mathcal{P}_1$ and a right module structure over $\mathcal{P}_2$ which commute.

\subsubsection{\sc Relative composition product and the relative inner Hom}\label{sectioncatalg}
Here we outline some categorical algebra in the category $\Vect(\k)_\Sigma$. We refer the reader to [R], [KM], [St] for more detail. 

Let $\mathcal{P}$ be an operad in $\Vect(\k)$. Then right $\mathcal{P}$-modules form a dg category over $\k$, denoted by $\Mod{-}\mathcal{P}$. We denote by $\Hom_{\Sigma,\Mod{-}\mathcal{P}}(M,N)$ the $\k$-vector space of morphisms of symmetric sequences respecting the right $\mathcal{P}$-module structure. 

For $X\in\Vect(\k)_\Sigma$, $Y\in \Mod{-}\mathcal{P}$, the composition product $X\circ Y$ is a right $\mathcal{P}$-module. 

Define the {\it relative inner hom} functor
$$
[-,-]_\mathcal{P}\colon (\Mod{-}\mathcal{P})^\op\times \Mod{-}\mathcal{P}\to\Vect(\k)_\Sigma
$$
as follows. The symmetric sequence $[M,N]_\mathcal{P}$ is defined as the equalizer
\begin{equation}
[M,N]_\mathcal{P}\to [M,N]\underset{v}{\overset{u}{\rightrightarrows}}[M\circ \mathcal{P},N]
\end{equation}
where, for $f\in [M,N]$, $u(f)=\mu_N(f\circ \mathcal{P})$, and $v(f)=f\mu_M^{\boxtimes n}$, so that
$$
[M,N]_\mathcal{P}(n)=\Hom_{\Sigma,\Mod{-}\mathcal{P}}(M^{\boxtimes n},N)
$$
There is a natural isomorphism 
\begin{equation}
\Hom_{\Sigma,\Mod{-}\mathcal{P}}(X\circ M,N)=\Hom_{\Sigma}(X,[M,N]_\mathcal{P})
\end{equation}
In this way, the category $\Mod{-}\mathcal{P}$ becomes a monoidal category tensored and enriched over the monoidal category $(\Vect(\k)_\Sigma,\circ,I)$, cf. [F1, Ch.1]. 

In particular, for $M\in \Mod{-}\mathcal{P}$, $[M,M]_\mathcal{P}$ is an operad. 
Moreover, for $M,N\in\Mod{-}\mathcal{P}$, $[M,N]_{\mathcal{P}}$ is an $[N,N]_\mathcal{P}{-}[M,M]_\mathcal{P}$-bimodule, and 
\begin{equation}
[\mathcal{P},\mathcal{P}]_\mathcal{P}=\mathcal{P}
\end{equation}
as an $\mathcal{P}{-}\mathcal{P}$-bimodule, and 
\begin{equation}
[\mathcal{P},M]_\mathcal{P}=M
\end{equation}
as right $\mathcal{P}$-module. 

One can define the {\it relative composition product} $M\circ_\mathcal{P}N\in \Vect(\k)_\Sigma$, for $M\in\Mod{-}\mathcal{P}, N\in\mathcal{P}{-}\Mod$, as the corresponding coequalizer. 

Moreover, for $M$ an $A{-}B$-bimodule, $N$ a $B{-}C$-bimodule, one defines $M\circ_BN\in A{-}\Mod{-}C$.
It gives rise to a functor 
\begin{equation}
-\circ_B-\colon A{-}\Mod{-}B\times B{-}\Mod{-}C\to A{-}\Mod{-}C
\end{equation}
One has the adjunction:
\begin{equation}
\Hom_{\Sigma, A{-}\Mod{-}C}(M\circ_BN,L)=\Hom_{\Sigma, A{-}\Mod{-}B}(M,[N,L]_C)
\end{equation}
It implies
\begin{lemma}\label{lemmabimod}
Let $A,B$ be operads, $X\in\Vect(\k)$ an $A{-}B$-bimodule. Then there is a canonical map of operads
\begin{equation}
\phi\colon A\to [X,X]_B
\end{equation}
\end{lemma}
Indeed, by the adjunction above, maps $A\to [X,X]_B$ are in 1-to-1 correspondence with the maps 
of right $B$-modules $A\circ_AX\to X$. We know that $A\circ_AX=X$, so the map $\phi$ is the map corresponding to the identity map of right $B$-modules.
\qed

\subsubsection{\sc The operadic shift}\label{sectionshift}
Let $X$ be a symmetric sequence in $\Vect(\k)$. Define another symmetric sequence $X\{1\}$, called {\it the operadic shift} of $X$, as
\begin{equation}
X\{1\}(n)=X[-n+1]\otimes_{\Sigma_n}\sgn_n
\end{equation}
where $\sgn_n$ denotes the (1-dimensional) sign representation of $\sigma_n$. 

We list the compatibility properties of the operadic shift with the monoidal structures and Hom's on $\Vect(\k)_\Sigma$:
\begin{lemma}\label{lemmashift}
Let $X,Y\in\Vect(\k)_\Sigma$. The following statements are true:
\begin{itemize}
\item[(i)] $(X\circ Y)\{1\}=X\{1\}\circ Y\{1\}$,
\item[(ii)] $[X,Y]\{1\}=[X\{1\},Y\{1\}]$,
\item[(iii)] $(X\otimes_\lev Y)\{1\}=(X\{1\})\otimes_\lev Y=X\otimes_\lev(Y\{1\})$,
\item[(iv)] $\Hom_\lev(X\{1\},Y\{1\})=\Hom_\lev(X,Y)$,
\end{itemize}
\end{lemma}
\begin{coroll}
Let $\mathcal{P}$ be an operad in $\Vect(\k)$. Then $\mathcal{P}\{1\}$ is again an operad. Similarly, $\mathcal{C}\{1\}$ is a cooperad as soon as $\mathcal{C}$ is a cooperad.
\end{coroll}

We denote
\begin{equation}
X\{n\}=(\dots((X\{1\})\{1\})\dots)\{1\}
\end{equation}
where the operation $-\{1\}$ is applied $n$ times. 

Let $X$ be a vector space considered as the symmetric sequence $X^{(0)}=(\underset{n=0}{X},0,0,\dots)$.
Then 
\begin{equation}
X^{(0)}\{n\}=(X[n])^{(0)}
\end{equation}
and
\begin{equation}
\End_\Op(X)\{1\}=[X^{(0)},X^{(0)}]\{1\}=\End_\Op(X[1])
\end{equation}
where $-[n]$ is the conventional shift of degree of a vector space.

\subsection{\sc Koszul operads}
We refer the reader to [GK] and [LV, Ch. 7] for theory of Koszul operads.

Here we just fix some notations.

For a quadratic operad $\mathcal{O}$ we denote by $\mathcal{O}^!$ the quadratic dual cooperad.
As well, we denote 
\begin{equation}\label{oash}
\mathcal{O}^\ash=(\mathcal{O}\{1\})^!=\mathcal{O}^!\{-1\}
\end{equation}

For any quadratic operad $\mathcal{O}$ there is a map of dg operads
\begin{equation}\label{barop}
\Bar_\Op(\mathcal{O}^\ash)\to \mathcal{O}
\end{equation}
inducing an isomorphism on $H^0$. A quadratic operad $\mathcal{O}$ is called {\it Koszul} if \eqref{barop} is a quasi-isomorphism of dg operads. (See [LV, Section 6.5] for the definition of $\Bar_\Op(\mathcal{O})$).

Many classical operads are Koszul, among them $\mathsf{Assoc}, \mathsf{Comm}, \mathsf{Lie},\mathsf{e}_n$. Their (shifted) Koszul dual are:
\begin{equation}
\begin{aligned}
\ & \mathsf{Assoc}^\ash=\mathsf{Assoc}^*\{-1\}\\
&\mathsf{Comm}^\ash=\mathsf{Lie}^*\{-1\}\\
&\mathsf{Lie}^\ash=\mathsf{Comm}^*\{-1\}\\
&\mathsf{e}_d^\ash=\mathsf{e}_d^*\{-d\}
\end{aligned}
\end{equation}

\subsection{\sc Convolution complexes}
\subsubsection{\sc The convolution complex for (co)operads}\label{section021}
Let $\mathcal{P}$ be an operad, and $\mathcal{C}$ a finite cooperad. Then $\Hom_\lev(\mathcal{C},\mathcal{P})$ is naturally an operad. Recall that its 
components are
\begin{equation}\label{homop}
\Hom_\lev(\mathcal{C},\mathcal{P})(n)=\Hom_\k(\mathcal{C}(n),\mathcal{P}(n))
\end{equation}
and the symmetric group $\Sigma_n$ acts on $\Hom_\Op(\mathcal{C},\mathcal{P})(n)$ as 
$$
(\sigma(f))(x)=\sigma(f(\sigma^{-1}x))
$$
The operad structure is given as
\begin{equation}\label{homop1}
\Hom_\lev(\mathcal{C},\mathcal{P})\circ\Hom_\lev(\mathcal{C},\mathcal{P})\to \Hom_\ev(\mathcal{C}\circ_1\mathcal{C},\mathcal{P}\circ\mathcal{P})\to \Hom_\lev(\mathcal{C},\mathcal{P})
\end{equation}
Here the first arrow is \eqref{strange3}, and the second arrow is given by the (co)operad structure maps $\mathcal{C}\to\mathcal{C}\circ_1\mathcal{C}$ and $\mathcal{P}\circ\mathcal{P}\to\mathcal{P}$.

We denote this operad structure on $\Hom_\lev(\mathcal{C},\mathcal{P})$ by $\Hom_\Op(\mathcal{C},\mathcal{P})$.

There is another construction, introduced in [KM], which associates a dg pre-Lie algebra $\mathcal{P}_\Sigma$ with a dg operad $\mathcal{P}$.

The underlying complex of $\mathcal{P}_\Sigma$ is defined as
\begin{equation}
\mathcal{P}_\Sigma=\bigoplus_{n\ge 1}\mathcal{P}(n)_{\Sigma_n}=\mathcal{P}\circ  \k^{(0)}
\end{equation}
(we assume that $\mathcal{P}$ is a non-unital operad, that is, $\mathcal{P}(0)=0$).

The formula for $\Psi_1\circ\Psi_2\in\mathcal{P}(m+n-1)$, for $\Psi_1\in \mathcal{P}(m)$ and $\Psi_2\in\mathcal{P}(n)$, reads:
\begin{equation}
\Psi_1\star \Psi_2=\sum_{i=1}^m\pm \Psi_1\circ_i \Psi_2
\end{equation}
where 
\begin{equation}\label{reli}
\Psi_1\circ_i \Psi_2=\Psi_1(\id^{\otimes (i-1)}\otimes\Psi_2\otimes\id^{\otimes (m-i)})
\end{equation}
The associated dg Lie algebra is also denoted by $\mathcal{P}_\Sigma$. For $\Psi_1,\Psi_2$ as above,
\begin{equation}
[\Psi_1,\Psi_2]=\Psi_1\star\Psi_2-(-1)^{|\Psi_1||\Psi_2|}\Psi_2\star\Psi_1
\end{equation}
See [MK, Sect. 1.7] for more detail. 

We can apply the above construction to the operad $\Hom_\Op(\mathcal{C},\mathcal{P})$, see \eqref{homop1}. We get the {\it operadic convolution Lie algebra} on the dg vector space 
\begin{equation}
\Conv(\mathcal{C},\mathcal{P})=\Hom_\Op(\mathcal{C},\mathcal{P})_\Sigma=\prod_{k\ge 1}\Hom_\lev(\mathcal{C}(k),\mathcal{P}(k))_{\Sigma_k}
\end{equation}

\comment
\subsubsection{\sc }
Consider the following situation. Let $\mathcal{P}_1,\mathcal{P}_2$ be operads, $\mathcal{M}$ a $(\mathcal{P}_1,\mathcal{P}_2)$-bimodule. 

Denote 
$$
\underline{\M}=\M\langle \k\rangle
$$
It follows from the associativity \eqref{asscirc} that $\underline{\M}$ is an algebra over $\mathcal{P}_1$.

Now the right $\mathcal{P}_2$-module structure on $\mathcal{M}$ makes $\underline{\M}$ a (right) module over the (pre-)Lie algebra $\underline{\mathcal{P}}_2$, as
\begin{equation}
\Theta\circ \Psi=\sum_{i=1}^m\Theta\circ_i \Psi \in \M(m+n-1)
\end{equation}
where $\Theta\in \M(m),\Psi\in \mathcal{P}_2(n)$, cf. [KM], ???.

\begin{lemma}\label{lemmasuper}
In the assumptions as above, the two structures on $\underline{M}$, the one of an algebra over $\mathcal{P}_1$ and the one of a right module over the Lie algebra $\underline{\mathcal{P}}_2$, are compatible, as follows:
\begin{equation}\label{commute}
[\Upsilon(x; m_1,\dots, m_k),\ell]=\sum_{i=1}^k\Upsilon(x; m_1,\dots, [m_i,\ell],\dots,m_k)
\end{equation}
where $x\in\mathcal{P}_1, m_1,\dots, m_k\in \underline{M}, \ell\in \underline{\mathcal{P}}_2$.
\end{lemma}
\begin{proof}
It follows directly from the definitions.

\qed
\end{proof}

\endcomment

\begin{remark}\label{remsept}{\rm
Assume an operad $\mathcal{P}$ in $\Vect(\k)$ acts on a (dg) symmetric sequence $X\in\Vect(\k)_\Sigma$. Then the operad $\mathcal{P}$ acts on the (dg) vector space $X_\Sigma$. Indeed, $X_\Sigma=X\circ \k^{(0)}$, and one has $$\mathcal{P}\circ (X\circ \k^{(0)})=(\mathcal{P}\circ X)\circ \k^{(0)}\xrightarrow{m_X\circ \id}X\circ \k^{(0)}$$
where $m_X\colon \mathcal{P}\circ X\to X$ defines the $\mathcal{P}$-action on $X$. 
}
\end{remark}

\subsubsection{\sc The convolution complex for (co)algebras}\label{section022}
Let $\mathcal{P}$ be an operad, $\mathcal{C}$ a finite cooperad.

A {\it coalgebra} $C$ over $\mathcal{C}$ is given by its structure map $\Delta_C\colon C\to \mathcal{C}\hat{\circ}_1 C$. A coalgebra $C$ is called {\it finite} if the map $\Delta_C$ factors as
$$
C\to\mathcal{C}\circ_1 C\to \mathcal{C}\hat{\circ}_1C
$$

 Let $A\in\Vect(\k)_\Sigma$ be an algebra over $\mathcal{P}$, and $C\in\Vect(\k)_\Sigma$ a finite coalgebra over $\mathcal{C}$.

Consider the symmetric sequence $\Hom_\lev(C,A)$. We claim that it becomes an algebra over the operad $\Hom_\Op(\mathcal{C},\mathcal{P})$.

Indeed, there are maps
\begin{equation}\label{ophomalg}
\Hom_\lev(\mathcal{C},\mathcal{P})\circ\Hom_\lev(C,A)\to
\Hom_\lev(\mathcal{C}\circ_1 C,\mathcal{P}\circ A)\to \Hom_\lev(C,A)
\end{equation}
The first map is given by \eqref{strange3}, and the second map is given by the compositions $C\to \mathcal{C}\circ_1 C$ and $\mathcal{P}\circ A\to A$.

\comment
Indeed, define
for $\Psi_1,\dots,\Psi_k\in\Hom(C,A)$, $\Theta\in \Hom_\Op(\mathcal{C},\mathcal{P})(k)$, 
$$
\Upsilon(\Theta\otimes (\Psi_1\otimes\dots\otimes \Psi_1))\in\Hom(C,A)
$$
as 
\begin{equation}
C\xrightarrow{\Delta}\mathcal{C}\circ C\xrightarrow{\Theta\otimes\Psi_\ldot}\mathcal{P}\circ A\xrightarrow{m}A
\end{equation}
where $\Psi_\ldot=\Psi_1\otimes \dots \otimes \Psi_k$, $\Delta$ is defined via the $\mathcal{C}$-coalgebra on $C$, and $m$ is defined via the $\mathcal{P}$-algebra structure on $A$.
\endcomment

There is a differential on $\Hom_\lev(C,A)$, defined for a homogeneous $\Psi$ as
\begin{equation}\label{eqdif1}
(d\Psi)(x)=d_A(\Psi(x))-(-1)^{|\Psi|}\Psi(d_C(x))
\end{equation}
where $d_A$ and $d_C$ are the differentials in $A$ and $C$, correspondingly.

When $\mathcal{P}$ is a dg operad and $\mathcal{C}$ is a dg cooperad, the operad $\Hom_\Op(\mathcal{C},\mathcal{P})$ is a dg operad, with the differential defined similarly to \eqref{eqdif1}, via the differentials in $\mathcal{C}$ and $\mathcal{P}$. 

We get:
\begin{lemma}\label{lca}
Let $\mathcal{P}$ be an operad, $\mathcal{C}$ a finite cooperad over $\Vect(\k)$. Let $A\in \Vect(\k)_\Sigma$ be an algebra over $\mathcal{P}$, $C\in\Vect(\k)_\Sigma$ a finite coalgebra over $\mathcal{C}$.
Then the construction above makes the symmetric sequence $\Hom_\lev(C,A)$ a dg algebra in $\Vect(\k)_\Sigma$ over the dg operad $\Hom_\Op(\mathcal{C},\mathcal{P})$. 
\end{lemma}
\qed

\begin{coroll}\label{corlca}
In the notations as in Lemma \ref{lca}, the dg operad $\Hom_{\Op}(\mathcal{C},\mathcal{P})$ acts on the dg vector space $\Hom_\Sigma(C,A)$
\end{coroll}
\begin{proof}
It follows from Lemma \ref{lca} and Remark \ref{remsept}.
\end{proof}

\subsubsection{\sc }\label{section0233}
Let $\mathcal{O}$ be a Koszul operad. Let $C$ be a coalgebra over $\mathcal{O}^\ash$, and $A$ an algebra over $\mathcal{O}$. Consider the symmetric sequence $\Hom_\lev(C,A)$. By \eqref{ophomalg}, it is an algebra over the operad $\Hom_\Op(\mathcal{O}^\ash,\mathcal{O})$. The following property is very important for the deformation theory:
\begin{lemma}\label{liekoszul}
For any Koszul operad $\mathcal{O}$ with finite-dimensional $\mathcal{O}(2)$, there is a map of operads
\begin{equation}
\phi\colon \mathsf{Lie}\{1\}\to \Hom_\Op(\mathcal{O}^\ash,\mathcal{O})
\end{equation}
\end{lemma}
The map $\phi$ was constructed in [GK] via the operadic analogues of Manin's black and white products of quadratic algebras, see [GK, 2.2], [LV, 8.8]. 

Later we use the following explicit formula for the arity component $\phi(2)$.

Let $E=\mathcal{O}(2)$, $E[1]=\mathcal{O}^\ash(2)$. Take any basis $t_1,\dots, t_m$ in $E$. Denote by $t_1[1],\dots,t_m[1]
$ the corresponding basis in $E[1]$. Then the map $\phi$ sends the canonical generator $\omega=[-,-]$ in $\mathsf{Lie}\{1\}(2)$ to
\begin{equation}\label{lieexplicit}
\phi(\omega)=\sum_{i=1}^mt_i[1]^*\otimes t_i\in\Hom_{\Sigma_2}(\mathcal{O}^\ash(2),\mathcal{O}(2))
\end{equation}

As the operad $\mathsf{Lie}$ is quadratic, the component $\phi(2)$ defines the map $\phi$.

\begin{coroll}
In the assumptions as above, the symmetric sequence $\Hom_\lev(C,A)\{-1\}$ has a natural structure of a Lie algebra. 
\end{coroll}
\qed

The symmetric sequence $\Hom_\lev(C,A)\{-1\}$ with its Lie algebra structure is often called {\it the convolution symmetric sequence}.

Consider an element $f\in\prod_{n\ge 1}\Hom_\lev(C,A)\{-1\}(n)^1$ satisfying the Maurer-Cartan equation:
\begin{equation}
d_0f+\frac12[f,f]=0
\end{equation}
In this situation, one can twist the differential in $\Hom_\lev(C,A)\{-1\}$ by $\ad(f)$, defined as $\ad(f)(-):=[f,-]$, where $[-,-]$ is the convolution Lie bracket, and get a dg Lie algebra
\begin{equation}
(\Hom_\Sigma (C,A)\{-1\}, d_0+\ad(f))
\end{equation}

\section{\sc The  bar-cobar adjunction and its relative version}\label{sectionbarcobar}

\subsection{\sc Pro-conilpotent $\mathcal{O}$-coalgebras}\label{subsectioncoalg}
Let $\mathcal{O}$ be an operad, $V\in\Vect(\k)_\Sigma$ a symmetric sequence. Denote by $\Alg(\mathcal{O})$ the category of non-unital $\mathcal{O}$-algebras in $\Vect(\k)_\Sigma$. 

The forgetful functor $\Alg(\mathcal{O})\to \Vect(\k)_\Sigma$, from $\mathcal{O}$-algebras in $\Vect(\k)_\Sigma$ to symmetric sequences,
admits a left adjoint. It is given by the free algebra over $\mathcal{O}$, generated by $V$:
\begin{equation}
\mathscr{F}_\mathcal{O}(V)=\mathcal{O}\circ V
\end{equation}
\begin{equation}
\Hom_{\Alg(\mathcal{O})}(\mathscr{F}_\mathcal{O}(V), A)=\Hom_{\Vect(\k)_\Sigma}(V,A)
\end{equation}
It is a non-unital $\mathscr{O}$-algebra.

There is a version of it for the unital augmented  $\mathcal{O}$-algebras. Denote by $\Alg^{u, aug}({\mathcal{O}})$ the category of such algebras over $\k$. For $A\in\Alg^{u,aug}({\mathcal{O}})$, define the functor 
\begin{equation}
R(A)=A_+=\Ker(\varepsilon\colon A\to \k^{(0)})
\end{equation}
where $\varepsilon \colon A\to \k^{(0)}$ is the augmentation, which is assumed to be a map of $\mathcal{O}$-algebras.  The functor $R$ admits a left adjoint:
\begin{equation}
\mathscr{F}_\mathcal{O}^u(V)=(\mathcal{O}\circ V)\oplus \k^{(0)}
\end{equation}
One has:
\begin{equation}
\Hom_{\Alg^{u, aug}(\mathcal{O})}(\mathscr{F}^u_\mathcal{O}(V), A)=\Hom_{\Vect(\k)_\Sigma}(V,A_+)
\end{equation}

Such a (right) adjoint functor to the forgetful functor does not exist, in general, for the coalgebras over an operad. One should restrict ourselves to a class of coalgebras, called {\it pro-conilpotent}. (In [T2], they are called pro-coartinian. In [Q, Appendix B] they are called connected.)
In this paper, we consider pro-conilpotent $\mathcal{C}$-coalgebras in {\it symmetric sequences}, where $\mathcal{C}$ is a {\it finite }cooperad, see \eqref{coopbis}.

For simplicity, assume that the finite cooperad $\mathcal{C}$ is {\it quadratic} (it is the only case we deal in this paper with). Assume that $\mathcal{C}$ is an {\it biaugmented} cooperad, where by a {\it biaugmentation} we mean operad maps $I\xrightarrow{\eta}\mathcal{C}\xrightarrow{\varepsilon}I$.  

Let $\mathcal{C}$ be a quadratic biaugmented cooperad. We denote 
$$
\mathcal{C}(n)_+=\Ker(\varepsilon\colon \mathcal{C}(n)\to I(n))
$$

By a {\it $\mathcal{C}$-coalgebra} we understand an object $C\in\Vect(\k)_\Sigma$ with a structure map  $\Delta\colon C\to \mathcal{C}\hat{\circ}_1 C$ and with a counit $\varepsilon_1 \colon C\to \k^{(0)}$, satisfying the natural axioms. By a coaugmentation of the coalgebra $C$ we mean a map $\eta_1\colon \k^{(0)}\to C$ of $\mathcal{C}$-coalgebras, such that $\varepsilon_1\circ\eta_1=\id$.

Set $F^0C=\mathrm{Im}(\eta_1)\simeq \k^{(0)}$.
The symmetric sequence $C/F^0C$ gets a structure of a non-counital $\mathcal{C}$-coalgebra.
Define an ascending filtration on $C/F^0C$, as follows. 

Set 
$$
F^i(C/F^0C)=\big\{x\in C/F^0C| \Delta(x)=(\underset{1}{0},\dots,\underset{i}{0},t_{i+1},t_{i+2},
\dots  \big\}
$$
where $t_\ell\in (\mathcal{C}(\ell)\otimes C^{\otimes \ell})^{\Sigma_\ell}$.

One has 
$$
0\subset F^1(C/F^0C)\subset F^2(C/F^0C)\subset F^3(C/F^0C)\subset\dots
$$

A coaugmented $\mathcal{C}$-coalgebra $C$ is called {\it pro-conilponent} if the ascending filtration $\{F^i(C/F^0C)\}_{i\ge 1}$ 
is {\it exhausted}.

Denote the category of pro-conilpotent counital coaugmented coalgebras over a cooperad $\mathcal{C}$ by $\Coalg_{\pronilp}^{cu,caug}(\mathcal{C})$.

Consider the forgetful functor $\mathscr{V}\colon \Coalg_{\pronilp}^{cu,caug}(\mathcal{C})\to\Vect(\k)_\Sigma$, defined as
\begin{equation}
\mathscr{V}(C)=C/F^0C
\end{equation}

One has:
\begin{lemma}\label{lemmacogen}
Let $\mathcal{C}$ be a finite quadratic biaugmented cooperad. Then the functor $\mathscr{V}\colon \Coalg_{\pronilp}^{cu,caug}(\mathcal{C})\to\Vect(\k)_\Sigma$ admits a right adjoint $\mathscr{F}$, given by the direct sum cofree coalgebra
\begin{equation}
\mathscr{F}^{\mathcal{C}}(V)=\big(\bigoplus_{n\ge 1}(\mathcal{C}(n)\otimes V^{\boxtimes n})^{\Sigma_n}\big)\oplus \k^{(0)}
\end{equation}
\end{lemma}
See e.g. [T2, Prop. 2.3] for a proof.

\qed

We refer to
$
\mathscr{F}^\mathcal{C}(V)$
as the {\it cofree $\mathcal{C}$-coalgebra cogenerated by $V$}.

One also has:
\begin{lemma}\label{lemmafincoalg}
Let $\mathcal{C}$ be a finite quadratic biaugmented cooperad, $C$ a pro-conilpotent coalgebra over $\mathcal{C}$.
Then the structure map 
$$
\Delta\colon C/F^0C\to \mathcal{C}\hat{\circ}_1 (C/F^0C)
$$
factors as
$$
C/F^0C\to \mathcal{C}\circ_1 (C/F^0C)\to \mathcal{C}\hat{\circ}_1(C/F^0C)
$$
where the second map is the canonical imbedding.
\end{lemma}

See [T2, Lemma 2.2] for a proof.

\qed

\subsection{\sc The bar dg symmetric sequence}\label{subsectionbar}
Let $X\in \Vect(\k)_\Sigma$ be a symmetric sequence, $\mathcal{O}$ a biaugmented Koszul operad with finite-dimensional components $\mathcal{O}(n)$. Denote by $\mathcal{O}^\ash$ the shifted Koszul dual cooperad, see \eqref{oash}.

Let $X\in\Vect(\k)_\Sigma$ be an $\mathcal{O}$-algebra. 
Then the cofree $\mathcal{O}^\ash$-coalgebra $\mathscr{F}^{\mathcal{O}^\ash}(V)$ is endowed with a component-wise differential, as follows.

Recall the twisting morphism $\kappa\colon \mathcal{O}^\ash\to \mathcal{O}$ of degree 1, see Lemma \ref{liekoszul}. One has the following composition
\begin{equation}
\mathcal{O}^\ash\circ_1 X\xrightarrow{{can}_{\mathcal{O}^\ash,X}}\mathcal{O}^\ash\circ  X\xrightarrow{\kappa\circ \id}\mathcal{O}\circ X\xrightarrow{{m_X}} X
\end{equation}
see \eqref{can} for the map $can$.

It is a map of symmetric sequences of degree +1. It can be extended to a coderivation of $\mathcal{O}^\ash\circ X$ as of $\mathcal{O}^\ash$-coalgebra, see 
[LV, 11.2.2]. All maps used in this extension are maps of symmetric sequences.

One denotes $d_\Bar$ the corresponding map. One has
$$
d_\Bar^2=0
$$

We denote by $\Bar_\mathcal{O}(X)$ the cofree coalgebra $\mathscr{F}_{\mathcal{O}^\ash}(X)$ over $\mathcal{O}^\ash$ endowed with this differential. 
The differential agrees with the cooperations by the Leibniz rule, and acts component-wise. As well, the differential commutes with the action of symmetric group(s).

It is a pro-conilpotent coalgebra over $\mathcal{O}^\ash$ in $\Vect(\k)_\Sigma$.

When $X$ is an arity 0 symmetric sequence, $\Bar_\mathcal{O}(X)$ is an arity 0 symmetric sequence as well, and agrees with the conventional definition.

\subsection{\sc The cobar dg symmetric sequence}
Let $Y$ be a pro-conilpotent $\mathcal{O}^\ash$-coalgebra in $\Vect(\k)_\Sigma$. One defines its cobar dg symmetric sequence as
\begin{equation}
\Cobar_{\mathcal{O}^\ash}(Y)=(\mathscr{F}_{\mathcal{O}}(Y/{F^0Y}),\ d_\Cobar)=(\mathcal{O}\circ (Y/F^0Y), d_\Cobar)
\end{equation}
where the cobar-differential $d_\Cobar$ is defined as follows. 

The symmetric sequence $\mathcal{O}\circ (Y/F^0Y)$ is an $\mathcal{O}$-algebra.
The differential is defined to satisfy the Leibniz rule, so it is defined by its restriction to the generators. Consider a map of symmetric sequences of degree +1
\begin{equation}
d_\Cobar\colon Y/F^0Y\to \mathcal{O}\circ (Y/F^0Y)
\end{equation}
defined as
\begin{equation}\label{ugugu}
Y/F^0Y\to \mathcal{O}^\ash\circ_1 (Y/F^0Y)\xrightarrow{{can}_{\mathcal{O}^\ash,Y/F^0Y}}\mathcal{O}^\ash\circ (Y/F^0Y)\xrightarrow{\kappa \circ\id}\mathcal{O}\circ (Y/F^0Y)
\end{equation}
Here the first map is given by the $\mathcal{O}^\ash$-coalgebra structure on $Y/F^0Y$ (inherited by the $\mathcal{O}^\ash$-coalgebra structure on $Y$).

One has
$$
d_\Cobar^2=0
$$

As for the case of the bar-differential, the extension of \eqref{ugugu} by Leiniz rule as an $\mathcal{O}$-algebra derivation, uses only maps of symmetric sequences. Therefore, the differential $d_\Cobar$ acts component-wise.

\subsection{\sc The bar-cobar adjunction}\label{barcobaradj}
Let $\mathcal{O}$ be a Koszul operad in $\Vect(\k)$, with finite-dimensional $\mathcal{O}(2)$, such that the cooperad $\mathcal{O}^\ash$ is bi-augmented, $X$ an $\mathcal{O}$-algebra, $Y$ a pro-conilpotent $\mathcal{O}^\ash$-coalgebra. Consider the dg vector space $\Hom_\Sigma(Y/F^0Y,X)$. The symmetric sequence $Y/F^0Y$ is a coalgebra over $\mathcal{O}^\ash$. Then it follows from Lemma \ref{lca}
that $\Hom_\Sigma(Y/F^0Y,X)$ is an algebra over the operad $\Hom_\lev(\mathcal{O}^\ash,\mathcal{O})$. By Lemma \ref{liekoszul}, there is a map of operads $\mathsf{Lie}\{1\}\to\Hom_\lev(\mathcal{O}^\ash,\mathcal{O})$. Therefore, by Remark \ref{remsept} (or by Corollary \ref{corlca}), the dg vector space $\Hom_\Sigma(Y/F^0Y,X)[-1]$ gets a (dg) Lie algebra structure.

\begin{prop}\label{propmainconj}
Let $\mathcal{O}$ be a Koszul operad with finite-dimensional components $\mathcal{O}(n)$, $X$ an $\mathcal{O}$-algebra, $Y$ a pro-conilpotent $\mathcal{O}^\ash$-coalgebra. Assume that the cooperad $\mathcal{O}^\ash$ is finite and biaugmented. One has the following functorial isomorphisms of sets:
\begin{equation}\label{eqsept1}
\Hom_{\Sigma,\Alg(\mathcal{O})}(\Cobar_{\mathcal{O}^\ash}(Y),X)=\MC(\Hom_\Sigma(Y/F^0Y,X)[-1])=\Hom_{\Sigma, \Coalg(\mathcal{O}^\ash)}
(Y,\Bar_\mathcal{O}(X))
\end{equation}
where $\Coalg(\mathcal{O}^\ash)$ stands for the category of pro-conilpotent $\mathcal{O}^\ash$-coalgebras.
\end{prop}
The set in the middle of \eqref{eqsept1} is called the set of {\it twisted morphisms}.

\begin{proof}
The proof repeats the well-known argument for the case when $X$ and $Y$ are dg vector spaces (considered as arity-zero symmetric sequences), see e.g. [LV, 11.3.1]. We make use the adjunction given by Lemma \ref{lemmacogen} (as well as its more straightforward counter-part for $\Alg(\mathcal{O})$). The compatibility with the (co)bar-differentials is translated to the Maurer-Cartan equation in $\Hom_\Sigma(Y/(F^0Y),X)[-1]$. 
\end{proof}

\subsection{\sc A fragment of the $\mathcal{P}$-relative bar-cobar adjunction}\label{sectionbarrel}
Recall that a symmetric sequence $Z$ is a right $\mathcal{P}$-module, where $\mathcal{P}$ is an operad, if there is a map 
$$
m_Z\colon Z\circ \mathcal{P}\to {Z}
$$
satisfying the natural associativity and unit axioms.

Let $Z$ be a pro-conilpotent coalgebra over a cooperad $\mathcal{C}$ and a right module over an operad $\mathcal{P}$.

We say that these two structures are compatible, if the following conditions are fulfilled:
\begin{itemize}
\item[(i)]
the right $\mathcal{P}$-module structure on $Z$ descents to a right $\mathcal{P}$-module structure on $Z/(F^0Z)$, that is
\begin{equation}
m_Z(F^0(Z)\circ \mathcal{P})\subset F^0(Z)
\end{equation}
\item[(ii)]
\begin{equation}\label{eqbizz}
\xymatrix{
(Z/F^0Z)\circ\mathcal{P}\ar[r]^{m_\mathcal{P}}\ar[d]_{\Delta_Z\circ\id}   &Z/F^0Z\ar[r]^{\Delta_Z}&\mathcal{C}\circ_1(Z/F^0Z)\ar[r]^{can}&\mathcal{C}\circ  (Z/F^0Z)\ar[d]^{=}\\
(\mathcal{C}\circ_1 (Z/F^0Z))\circ \mathcal{P}\ar[r]^{can}   &(\mathcal{C}\circ (Z/F^0Z))\circ \mathcal{P}\ar[r]&\mathcal{C}\circ ((Z/F^0Z)\circ\mathcal{P})\ar[r]&C\circ (Z/F^0Z)
}
\end{equation}
Here $can$ denotes the canonical isomorphism \eqref{can0}.

\end{itemize}

We denote by $\mathscr{C}(\mathcal{C},\mathcal{P})$ the category of symmetric sequences $Z$ with the above conditions. 

For any right $\mathcal{P}$-module $Z$ and any cooperad $\mathcal{C}$, the cofree coalgebra $\mathscr{F}^\mathcal{C}(Z)$ is naturally an object of $\mathscr{C}(\mathcal{C},\mathcal{P})$.

One has:
\begin{lemma}
Let $\mathcal{O}$ be a Koszul operad, $\mathcal{P}$ an operad, $X$ a $\mathcal{O}{-}\mathcal{P}$-bimodule. Then the bar dg symmetric sequence $\Bar_\mathcal{O}(X)$ is a dg object of the category $\mathscr{C}(\mathcal{O}^\ash,\mathcal{P})$, what amounts to say that the bar-differential is compatible with the right $\mathcal{P}$-action.
\end{lemma}
\qed

For two right $\mathcal{P}$-modules $M,N\in\Vect(\k)_\Sigma$, define $\Hom_{\lev, \Mod{-}\mathcal{P}}(M,N)\in\Vect(\k)_\Sigma$ as the equalizer
\begin{equation}
\Hom_{\lev, \Mod{-}\mathcal{P}}(M,N)\to\Hom_{\lev}(M,N)\underset{v}{\overset{u}{\rightrightarrows}}\Hom_\lev(M\circ \mathcal{P},N)
\end{equation}
where, for $f\in \Hom_\lev(M,N)$,  $u(f)$ is defined as the composition $M\circ \mathcal{P}\xrightarrow{\mu_M} M\xrightarrow{f}N$, and $v(f)$ is defined as the composition $M\circ \mathcal{P}\xrightarrow{f\circ\id} N\circ \mathcal{P}\xrightarrow{\mu_N} N$. 

One easily shows that
\begin{equation}\label{eqsept10}
\Hom_{\Sigma,\Mod{-}\mathcal{P}}(M,N)=\Hom_{\lev,\Mod{-}\mathcal{P}}(M,N)\circ\k^{(0)}
\end{equation}
(the functor of invariants is right adjoint and thus commutes with the limits, and in the case of $\cchar \k=0$ the spaces of invariants and coinvariants are canonically isomorphic).

\vspace{2mm}

Let $\mathcal{O}$ be a Koszul operad, $\mathcal{P}$ an operad, $X$ an $\mathcal{O}$-$\mathcal{P}$-bimodule, $Y\in \mathscr{C}(\mathcal{O}^\ash,\mathcal{P})$.

Below we construct a dg Lie algebra structure on $\Hom_{\Sigma,\Mod{-}\mathcal{P}}(Y/\k,X)[-1]$. (For the case $\mathcal{P}=I$, this Lie algebra coincides with the one which figures in the middle term of \eqref{eqsept1}).

To this end, we construct an operad action 
$$
\Hom_\lev(\mathcal{O}^\ash,\mathcal{O})\circ \Hom_{\lev,\Mod{-}\mathcal{P}}(Y/F^0Y,X)\to \Hom_{\lev,\Mod{-}\mathcal{P}}(Y/F^0Y,X)
$$
For a general functor $F\colon \mathscr{C}\to\mathscr{D}$, and a diagram $D\colon I\to\mathscr{C}$, there is a canonical morphism in $\mathscr{D}$:
\begin{equation}
\xi\colon F(\lim D(i))\to \lim F(D(i))
\end{equation}
Therefore, one has a morphism in $\Vect(\k)_\Sigma$:
\begin{equation}\label{eqsept12}
\begin{aligned}
\ &\xi\colon \Hom_\lev(\mathcal{O}^\ash, \mathcal{O})\circ \Hom_{\lev,\Mod{-}\mathcal{P}}(Y/F^0Y,X)\to \\
&\lim\Big(\Hom_\lev(\mathcal{O}^\ash,\mathcal{O})\circ\Hom_\lev( Y/F^0Y,X)\underset{\id\circ v}{\overset{\id\circ u}{\rightrightarrows}}\Hom_\lev(\mathcal{O}^\ash,\mathcal{O})\circ\Hom_\lev( (Y/F^0Y)\circ\mathcal{P},X)\Big)
\end{aligned}
\end{equation}
The limit in the r.h.s. of \eqref{eqsept12} is further mapped to
\begin{equation}
\begin{aligned}
\ &\lim\Big(\Hom_\lev(\mathcal{O}^\ash,\mathcal{O})\circ\Hom_\lev( Y/F^0Y,X)\underset{\id\circ v}{\overset{\id\circ u}{\rightrightarrows}}\Hom_\lev(\mathcal{O}^\ash,\mathcal{O})\circ\Hom_\lev( (Y/F^0Y)\circ\mathcal{P},X)\Big)\overset{\footnotesize{\eqref{strange3}}}{\to}\\
&\lim\Big(\Hom_\lev(\mathcal{O}^\ash\circ (Y/F^0Y),\mathcal{O}\circ X)\underset{v_1}{\overset{u_1}{\rightrightarrows}}
\Hom_\lev(\mathcal{O}^\ash\circ( (Y/F^0Y)\circ \mathcal{P}),\mathcal{O}\circ X)\Big){=}\\
&\lim\Big(\Hom_\lev(\mathcal{O}^\ash\circ (Y/F^0Y),\mathcal{O}\circ X)\underset{v_2}{\overset{u_2}{\rightrightarrows}}
\Hom_\lev((\mathcal{O}^\ash\circ(Y/F^0Y))\circ \mathcal{P}),\mathcal{O}\circ X)\Big)\overset{*}{\to}\\
&\lim\Big(\Hom_\lev(Y/F^0Y,X)\underset{v}{\overset{u}{\rightrightarrows}}\Hom_\lev((Y/F^0Y)\circ\mathcal{P},X)\Big)=\\
&\Hom_{\lev,\Mod{-}\mathcal{P}}(Y/F_0Y,X)
\end{aligned}
\end{equation}
(the arrows $u_1,v_1$ and $u_2,v_2$ are clear, and we skip their definitions).
Here the crucial point is the arrow marked by $*$ in the third line. It encodes both diagram \eqref{eqbizz} (expressing that $Y\in\mathscr{C}(\mathcal{O}^\ash,\mathcal{P})$), and the property of $X$ being an $\mathcal{O}{-}\mathcal{P}$-bimodule. Indeed, 
the equality follows from commutativity of the two diagrams, corresponded to the cases $\omega=u$ and $\omega=v$ of the diagram below:
\begin{equation}\label{eqsept15}
\xymatrix{
\Hom_\lev(\mathcal{O}^\ash\circ (Y/F^0Y),\mathcal{O}\circ X)\ar[r]^{\omega_2\hspace{5mm}}\ar[d]&\Hom_\lev((\mathcal{O}^\ash\circ (Y/F^0Y))\circ\mathcal{P},\mathcal{O}\circ X)\ar[d]\\
\Hom_\lev(Y/F^0Y,X)\ar[r]^{\omega}&\Hom_\lev((Y/F^0Y)\circ \mathcal{P},X)
}
\end{equation}
The commutativity of \eqref{eqsept15} follows from \eqref{eqbizz}, whereas its commutativity for $\omega=v$ follows from the fact that $X$ is an $\mathcal{O}{-}\mathcal{P}$-bimodule.

One easily shows that the constructed map of symmetric sequences 
$$
\Hom_\lev(\mathcal{O}^\ash,\mathcal{O})\circ \Hom_{\lev,\Mod{-}\mathcal{P}}(Y/F^0Y,X)\to \Hom_{\lev,\Mod{-}\mathcal{P}}(Y/F^0Y,X)
$$
gives rise to an action of the operad $\Hom_\lev(\mathcal{O}^\ash,\mathcal{O})$ (see Corollary \ref{corollbm}) on $\Hom_{\lev,\Mod{-}\mathcal{P}}(Y/F^0Y,X)\in\Vect(\k)_\Sigma$.

\vspace{2mm}

\comment

The symmetric sequence $\Hom_\lev(Y/F^0Y,X)$ has a left action of the operad 
$\Hom_\Op(\mathcal{O}^\ash,\mathcal{O})$, and one has a map of operads $\mathsf{Lie}\{1\}\to\Hom_\Op(\mathcal{O}^\ash,\mathcal{O})$ by Lemma \ref{liekoszul}. 
This left action of $\Hom_\Op(\mathcal{O}^\ash,\mathcal{O})$ preserves the left $\mathcal{P}$-equivariant maps.
Indeed, one has
\begin{equation}
\begin{aligned}
\ &(Y/F^0Y)\circ \mathcal{P}\to \mathcal{O}^\ash\circ_1( (Y/F^0Y)\circ \mathcal{P})=\mathcal{O}^\ash\circ((Y/F^0Y)\circ \mathcal{P})=(\mathcal{O}^\ash\circ (Y/F^0Y))\circ \mathcal{P}=\\
&\mathcal{O}^\ash\circ((Y/F^0Y)\circ \mathcal{P})\xrightarrow{*}\mathcal{O}\circ (X\circ \mathcal{P})\to
(\mathcal{O}\circ X)\circ \mathcal{P}\to X\circ \mathcal{P}
\end{aligned}
\end{equation}
We see that the $\mathcal{P}$-equivariantity of all arrows $(Y/F^0Y)\to X$ in the arrow marked by $*$ implies the $\mathcal{P}$-equivariantity of the total map.

In the first equality, we identified $-\circ-$ with $-\circ_1-$, see \eqref{can0} and the discussion thereafter. 

\endcomment

One has a map of operads $\mathsf{Lie}\{1\}\to\Hom_\Op(\mathcal{O}^\ash,\mathcal{O})$, by Lemma \ref{liekoszul}. Therefore, \\
$\Hom_{\lev, \Mod{-}\mathcal{P}}(Y/F^0Y,X)$ becomes a $\mathsf{Lie}\{1\}$-algebra. The same is true for \\
$\Hom_{\lev, \Mod{-}\mathcal{P}}(Y/F^0Y,X)\circ \k$, see Remark \ref{remsept}. The latter space is identified with\\ $\Hom_{\Sigma, \Mod{-}\mathcal{P}}(Y/F^0Y,X)$, see  \eqref{eqsept10}.

It is an algebra over $\mathsf{Lie}\{1\}$. Therefore, $\Hom_{\Sigma,\Mod{-}\mathcal{P}}(Y/F^0Y,X)[-1]$ is a dg  Lie algebra (the differential is $d_0$, that is, it comes from the inner differentials on $X$ and $Y$).

Now we can formulate our statement:
\begin{prop}\label{baradjrel}
Let $\mathcal{O}$ be a Koszul biaugmented operad, $\mathcal{P}$ an operad, $X$ a $\mathcal{O}{-}\mathcal{P}$-bimodule, $Y$ an object of the category $\mathscr{C}(\mathcal{O}^\ash,\mathcal{P})$. One has:
\begin{equation}
\MC(\Hom_{\Sigma,\Mod{-}\mathcal{P}}({Y}/(F^0Y),X)[-1])=\Hom_{\Sigma,\mathscr{C}(\mathcal{O}^\ash,\mathcal{P})}(Y,\Bar_\mathcal{O}(X))
\end{equation}
\end{prop}
\begin{proof}
After all preparations, the statement becomes almost trivial. A map $t\colon Y\to \Bar_\mathcal{O}(X)$ of the underlying coalgebras  is the same that a map $t^\prime\colon Y/F^0Y\to X$, by Lemma \ref{lemmacogen}. The map $t$ is a map of right $\mathcal{P}$-modules if and only if  the map $t^\prime$ is. Finally, the compatibility of $t$ with the bar-differential results in the Maurer-Cartan equation on $t^\prime$. 
\end{proof}

\section{\sc Deformation theory of a morphism of operads $\mathsf{e}_n\to\mathcal{P}$ with a cocommutative base}\label{defcomm}

\subsection{\sc The case $\mathcal{P}=[X,X]$}\label{defcomm1}

\subsubsection{\sc }\label{defcomm11}
Let $V\in\Vect(\k)$ be a (dg) vector space over $\k$, $X\in\Vect(\k)_\Sigma$ a symmetric sequence.
Recall the symmetric sequence $V\star X:=V^{(0)}\boxtimes X$ (that is, $(V\star X)(n)=V\otimes X(n)$).

One has the adjunction 
\begin{equation}\label{adjstar}
\Hom_{\Sigma}(V\star X,Y)=\Hom_\k(V,\Hom_{\Sigma}(X,Y))
\end{equation}
One has the following fact:
\begin{lemma}\label{lemmastarprod}
Let $V\in \Vect(\k)$, $X\in\Vect(\k)_\Sigma$. The following statements are true:
\begin{itemize}
\item[(i)] Let $\mathcal{P}_1,\mathcal{P}_2$ are operads, $V$ is an algebra over $\mathcal{P}_1$, $X$ an algebra over $\mathcal{P}_2$. Then $V\star X$ is an algebra over the operad $\mathcal{P}_1\otimes_\lev \mathcal{P}_2$.
\item[(ii)] Let $\mathcal{C}_1,\mathcal{C}_2$ are cooperads, $V$ is a coalgebra over $\mathcal{C}_1$, $X$ is a coalgebra $\mathcal{C}_2$. Then $V\star X$ is a coalgebra over $\mathcal{C}_1\otimes_\lev\mathcal{C}_2$.
\end{itemize}
\end{lemma}
\begin{proof}
\

(i): One has maps $$\mathcal{P}_1(k)\otimes_{\Sigma_k}V^{\otimes k}\to V$$ and 
$$\mathcal{P}_2(k)\otimes_{\Sigma_k}\Ind_{\Sigma_{n_1}\times\dots\times\Sigma_{n_k}}^{\Sigma_{n_1+\dots+n_k}}(X(n_1)\otimes\dots\otimes X(n_k))\to X(n_1+\dots+n_k)$$
They give
$$
(\mathcal{P}_1(k)\otimes_\lev \mathcal{P}_2(k))\otimes_{\Sigma_k}\Ind_{\Sigma_{n_1}\times\dots\times\Sigma_{n_k}}^{\Sigma_{n_1+\dots+n_k}}(V^{\otimes k}\otimes X(n_1)\otimes \dots\otimes X(n_k))\to V\otimes X_{n_1+\dots+n_k}
$$
which gives a structure of an algebra over $\mathcal{P}_1\otimes_\lev\mathcal{P}_2$ on $V\star X$.

(ii): is analogous.
\end{proof}

\subsubsection{\sc}\label{defcomm12}
Let $\mathcal{O}$ be a Koszul operad, $X\in \Vect(\k)_\Sigma$ an $\mathcal{O}$-algebra. Recall the dg sequence $$\Bar_\mathcal{O}(X)=(\mathcal{O}^\ash\circ X,d_0+d_\Bar)$$
see Section \ref{subsectionbar}. It is a coalgebra over the cooperad $\mathcal{O}^\ash$. 

Let $X,Y\in\Vect(\k)_\Sigma$ be $\mathcal{O}$-algebras, $f\colon X\to Y$ a map of $\mathcal{O}$-algebra. It defines a map 
$$
\Bar(f)\colon \Bar_\mathcal{O}(X)\to\Bar_{\mathcal{O}}(Y)
$$
of dg symmetric sequences of coalgebras over $\mathcal{O}^\ash$.

We prefer to deal with coalgebras over finite biaugmented cooperads, as the concept of a pro-conilpotent coalgebra over a cooperad, see Section \ref{subsectioncoalg}, necessary for the adjunction in Proposition \ref{propmainconj}, exists only for such cooperads. 
In our case, $\mathcal{O}=\mathsf{e}_n$, the cooperad $\mathsf{e}_n^\ash$ is not biaugmented.
On the other hand, $\mathsf{e}_n^\ash=\mathsf{e}_n^*\{-n\}$, and the cooperad $\mathsf{e}_n^*$ is biaugmented.
Therefore, we replace the map $\Bar(f)$ of $\mathsf{e}_n^!$-coalgebras by the corresponding map of $\mathsf{e}_n^\ash\{n\}=e_n^*$-coalgebras.

Consider
$$
\Bar(f)\{n\}\colon\Bar_{\mathsf{e}_n^\ash}(X)\{n\}\to \Bar_{\mathsf{e}_n^\ash}(Y)\{n\}
$$
One has:
\begin{equation}\label{barshifted1}
(\mathsf{e}_n^\ash\circ X)\{n\}=(\mathsf{e}_n^\ash\{n\})\circ (X\{n\})=\mathsf{e}_n^*\circ (X\{n\})
\end{equation}
Denote
\begin{equation}\label{barshifted2}
\Bar_n(X)=(\mathsf{e}_n^*\circ (X\{n\}), d_\Bar\{n\})
\end{equation}
By abuse of notations, we will use the notation $d_\Bar$ for both the shifted differential $d_\Bar\{n\}$ and the original bar differential.

Let $a$ be a coalgebra over $\mathsf{Comm}^*$. By Lemma \ref{lemmastarprod},
$$
a\star \Bar_n(X)
$$
is a coalgebra over $\mathsf{Comm}^*\otimes \mathsf{e}_n^*=\mathsf{e}_n^*$. 

Let $f\colon X\to Y$ be as above.
Denote by
$$\Coalg=\Coalg(\mathsf{Comm}^*)$$ 
the category of pro-conilpotent cocommutative coalgebras over $\k$. Define the functor 
$$
F_{X,Y}^f\colon \Coalg\to\Sets
$$,
as follow:
\begin{equation}\label{funf}
F_{X,Y}^f(a)=\Big\{\phi\in \Hom_{\Sigma,\Coalg(\mathsf{e}_n^*)}\big(a\star\Bar_n(X),\Bar_n(Y)\big), \ \ \phi\circ \eta=\Bar(f)\{n\}\Big\}
\end{equation}
Here $\eta\colon \k\to a$ is the coaugmentation of $a$. 

Our first task is to show that this functor is representable.

\subsubsection{\sc}
The following Proposition easily follows from the bar-cobar duality in symmetric sequences, proven in Section \ref{barcobaradj}.
\begin{prop}\label{propbarcobarbis}
Let $X\in \Vect(\k)_\Sigma$ be a $\mathsf{e}_n$-algebra, and let $C\in\Vect(\k)_\Sigma$ be a pro-conilpotent $\mathsf{e}_n^*$-coalgebra. One has:
\begin{equation}\label{barcobarbis}
\Hom_{\Sigma,\Coalg(\mathsf{e}_n^*)}(C,\Bar_{n}(X))=\MC\Big(\Hom_{\Sigma}(C/F^0C, X\{n\})[-1]\Big)
\end{equation}
where $\Coalg(\mathsf{e}_n^*)$ stands for the category of pro-conilpotent $\mathsf{e}_n^*$-coalgebras.
\end{prop}
\begin{proof}
It follows from Proposition \ref{propmainconj} for $\mathcal{O}=\mathsf{e}_n\{n\}$.
\end{proof}

\subsubsection{\sc}
By Proposition \ref{propbarcobarbis}, one has:
\begin{equation}\label{funf2}
F^f_{X,Y}(a)=\Big\{\theta\in \MC\big(\Hom_{\Sigma}((a\star \Bar_n(X))/\k,Y\{n\})[-1]\big),\ \theta\circ \eta=\Bar(f)_\pr\Big\}
\end{equation}
where
\begin{equation}\label{funf3}
\Bar(f)_\pr\in \MC(\Hom_\Sigma(\Bar_n(X\{n\})/\k,Y\{n\})[-1])
\end{equation}
is the element corresponded to $\Bar(f)$ by \eqref{barcobarbis}.

One further has:
\begin{equation}\label{funf4}
F^f_{X,Y}(a)=\Big\{ \theta\in \MC\big(\Hom_\k(a,\Hom_\Sigma(\Bar_n(X\{n\})/\k,Y\{n\})[-1])\big), \theta\circ \eta=\Bar(f)_\pr\Big\}
\end{equation}
We made use the adjunction \eqref{adjstar}.

Consider the graded vector space 
$$
\Hom_\Sigma(\Bar_n(X)/\k,Y\{n\})=
\Hom_{\Sigma}(\mathsf{e}_n^*\circ X\{n\},Y\{n\})
$$
By Lemma \ref{lca}, it is an algebra over the operad $\Hom_\Op(\mathsf{e}_n^*,\mathsf{e}_n\{n\})=\Hom_\Op(\mathsf{e}_n^\ash,\mathsf{e}_n)$. 

By Lemma \ref{liekoszul}, there is a map of operads
\begin{equation}
\mathsf{Lie}\{1\}\to \Hom_\Op(\mathcal{O}^\ash,\mathcal{O})
\end{equation}
for any Koszul operad $\mathcal{O}$.

It makes $\Hom_\Sigma(\mathcal{O}^\ash\circ X,Y)[-1]$ is an algebra over the same operad $\mathsf{Lie}$. 
In our situation, it implies that $\Hom_\Sigma(\mathsf{e}_n^*\circ X\{n\},Y\{n\})[-1]$ is a graded Lie algebra.
One easily sees that the inner differential $d_0$ (which is equal to 0 provided the differentials on $X, Y$ are 0), and the differential $d_\Bar$, are compatible with this Lie algebra structure.

Consider the dg Lie algebra
\begin{equation}
\begin{aligned}
\Def_0(X,Y)=&\big(\Hom_\Sigma(\Bar_n(X)/\k,Y\{n\})[-1], d=d_0+d_\Bar\big)\\
=&\big(\Hom_{\Sigma}(\mathsf{e}_n^*, [X\{n\},Y\{n\}])[-1], d=d_0+d_\Bar^\sim\big)
\end{aligned}
\end{equation}
where $d_\Bar^\sim$ is the differential corresponded to $d_\Bar$ by the adjunction, and $[X\{n\},Y\{n\}]=[X,Y]\{n\}$ by Lemma \ref{lemmashift}.

The element $\Bar(f)_\pr$ is a degree 1 Maurer-Cartan element in $\Def_0(X,Y)$. Denote 
\begin{equation}
d_f=\ad(\Bar(f)_\pr)
\end{equation}
and define
\begin{equation}
\begin{aligned}
\Def(X\xrightarrow{f}Y)=&\big(\Hom_\Sigma(\Bar_n(X)/\k,Y\{n\})[-1], d=d_0+d_\Bar+d_f\big)\\
=&\big(\Hom_\Sigma(\mathsf{e}_n^*,[X,Y]\{n\})[-1],d=d_0+d_\Bar^\sim+d_f^\sim\big)
\end{aligned}
\end{equation}
It is a dg Lie algebra. 

\subsubsection{\sc}\label{defcomm12bis}
One has:
\begin{prop}\label{propreprbasic}
Let $X,Y$ be $\mathsf{e}_n$-algebras in $\Vect(\k)_\Sigma$, and let $f\colon X\to Y$ be a map of $\mathsf{e}_n$-algebras. The functor
$F_{X,Y}^f$ is representable, by the dg coalgebra
\begin{equation}
a_\rep={C}_{\CE}(\Def(X\xrightarrow{f}Y),\k)
\end{equation}
which is the Chevalley-Eilenberg chain complex of the dg Lie algebra $\Def(X\xrightarrow{f}Y)$.
\end{prop}

\begin{proof}
(Cf. [T2, Prop. 3.2]). By \eqref{funf4}, one has
\begin{equation}\label{funf7}
F^f_{X,Y}(a)=\Big\{ \theta\in \MC\big(\Hom_\k(a,\Hom_\Sigma(\Bar_n(X)/\k,Y\{n\})[-1])\big), \theta\circ \eta=\Bar(f)_\pr\Big\}
\end{equation}
Denote
\begin{equation}
\begin{aligned}
\ &L_0=\big(\Hom_\Sigma(\Bar_n(X)/\k,Y\{n\})[-1],\ d=d_0+d_\Bar)\\
&L_a=\big(\Hom_\k(a,\Hom_\Sigma(\Bar_n(X)/\k,Y\{n\})[-1]),\ d=d_0+d_a+d_\Bar\big)
\end{aligned}
\end{equation}
where $d_a$ is the component coming from the differential on the dg coalgebra $a$.

There are maps of coalgebras 
\begin{equation}\label{funf8}
\k\xrightarrow{\eta} a\xrightarrow{\varepsilon}\k,\ \ \varepsilon\circ\eta=\id
\end{equation}
where $\varepsilon$ is the counit, and $\eta$ is the coaugmentation. 

They induce maps of dg Lie algebras:
\begin{equation}\label{funf9}
L_0\xrightarrow{\varepsilon_*}L_a
\xrightarrow{\eta_*}L_0
\end{equation}
such that $\eta_*\circ\varepsilon_*=\id$. 

By \eqref{funf7}, we are interested in 
\begin{equation}\label{funf10}
F_{X,Y}^f(a)=\{\theta\in \MC(L_a),\ \eta_*(\theta)=\Bar(f)_\pr\}
\end{equation}
where $\Bar(f)_\pr$ is considered as a solution of the Maurer-Cartan equation in $L_0$.

We identify this MC solution in $L_0$ with $\theta_0=\varepsilon_*(\Bar(f)_\pr)$, regarded as a Maurer-Cartan element in $L_a$. 

Then any Maurer-Cartan element in \eqref{funf10} is of form
\begin{equation}\label{funf11}
\theta=\theta_0+\xi
\end{equation}
where
\begin{equation}\label{funf11}
\xi\in \Ker(\eta_*)=\Hom_\k(a/\k,\Hom_\Sigma(\Bar_n(X),Y\{n\})[-1])=L_a^\prime
\end{equation} 
We want to rewrite the Maurer-Cartan equation on $\theta\in L_a$ as some equation on $\xi\in L_a^\prime$.

The Maurer-Cartan equation on $\theta$ is 
\begin{equation}
d\theta+\frac12[\theta,\theta]=0\ \Leftrightarrow d_1\xi+[\Bar(f)_\pr,\xi]+\frac12[\xi,\xi]=0
\end{equation}
where $d_1=d_0+d_a+d_\Bar$.

The conclusion is that the Maurer-Cartan equation on $\theta$ is the same that the Maurer-Cartan equation on $\xi$ in the dg Lie algebra
\begin{equation}
L_a^{\prime\prime}=\big(\Hom_\k(a/\k,\Hom_\Sigma(\Bar_n(X),Y\{n\})[-1]),\ d=d_0+d_a+d_\Bar+\ad(\Bar(f)_\pr)\big)
\end{equation}
Note that, as a dg Lie algebra,
\begin{equation}
L_a^{\prime\prime}=\Hom_\k(a/\k,\Def(X\xrightarrow{f}Y))
\end{equation}
Then the adjunction in Proposition \ref{propmainconj} gives
\begin{equation}
F_{X,Y}^f(a)=\{\xi\in \MC(\Hom_\k(a/\k,\Def(X\xrightarrow{f}Y)))\}=\Hom_{\Coalg(\mathsf{Comm})}(a,{C}_\CE(\Def(X\xrightarrow{f}Y),\k))
\end{equation}
\end{proof}

\subsubsection{\sc }\label{defcomm13}
Now turn back to our original definition of the functor $F_{X,Y}^f$ as \eqref{funf}. It follows immediately that, for a chain of maps of $\mathsf{e}_n$-algebras in $\Vect(\k)_\Sigma$
\begin{equation}
X\xrightarrow{f}Y\xrightarrow{g}Z
\end{equation}
there is a map of sets
\begin{equation}
F^g_{Y,Z}(a^\prime) \times F_{X,Y}^f(a)\to F^{gf}_{X,Z}(a\otimes a^\prime)
\end{equation}
what gives rise to a map of bifunctors $\Coalg\times \Coalg\to\Sets$:
\begin{equation}
F^g_{Y,Z}(-_1)\times F^f_{X,Y}(-_2)\to F^{gf}_{X,Z}\circ \bigotimes(-_1,-_2)
\end{equation}
Plugging in the representative objects for $-_1$ and $-_2$ and the identity maps, one gets a map of dg cocommutative coalgebras:
\begin{equation}\label{transition}
T_{f,g}\colon{C}_\CE(\Def(Y\xrightarrow{g}Z),\k)\bigotimes{C}_\CE(\Def(X\xrightarrow{f}Y),\k)\to{C}_\CE(\Def(X\xrightarrow{gf}Z),\k)
\end{equation}

These maps of dg cocommutative coalgebras are associative:
\begin{equation}
T_{gf,h}\circ (\id\otimes T_{f,g})=T_{f,hg}\circ (T_{g,h}\otimes \id)
\end{equation}

\subsubsection{\sc }\label{defcomm14}
Plug $X=Y=Z$, $f=g=\id$ to \eqref{transition}. 

It follows that $A={C}_\CE(\Def(X\xrightarrow{\id}X),\k)$ is a monoid object in the category cocommutative dg coalgebras. That is,
it is a cocommutative dg bialgebra. 

It follows from the Milnor-Moore theorem [Q, Appendix B] that, as a graded Hopf algebra,  $A=\mathscr{U}(\mathfrak{g})$, where $\mathfrak{g}$ is the Lie algebra of primitive elements. In our case,
\begin{equation}
\mathfrak{g}=\Def(X\xrightarrow{\id}X)[1]
\end{equation}
Thus, the product \eqref{transition} gives {\it some} Lie algebra structure on 
\begin{equation}
\Def(X\xrightarrow{\id}X)[1]=\prod_{\ell\ge 1}\Hom_{\Sigma_\ell}((\mathsf{e}_n^*\circ X\{n\})(\ell),X\{n\}(\ell))
\end{equation}
It is a Gerstenhaber-like Lie bracket, which has been obtained from the convolution Lie bracket.

\begin{lemma}\label{lemmazero1}
The Lie bracket on $\Def(X\xrightarrow{\id}X)$, defined in Section \ref{defcomm12}, is equal to 0. 
\end{lemma}
\begin{proof}
We know that ${C}_\CE(\Def(X\xrightarrow{\id}X,\k)$ is a dg Hopf algebra. It is of the form $\mathscr{U}(\mathfrak{g})$, where $\mathfrak{g}=\Def(X\xrightarrow{\id}X)[1]$ is the space of primitive elements. 
The Chevalley-Eilenberg chain differential is equal to 0 on $\mathfrak{g}$. As a conclusion, we get that the Chevalley-Eilenberg differential on any power $S^k(\Def(X\xrightarrow{\id}X)[1])$ is 0. It is equivalent to saying that $\Def(X\xrightarrow{\id}X)$ is an abelian Lie algebra.

\end{proof}

\subsubsection{\sc}

\begin{lemma}
The Lie bracket on $\Def(X\xrightarrow{\id}X)[1]$ given  by the Milnor-Moore theorem is identified with the bracket given by the commutator of coderivativations  of the cofree $\mathsf{e}_n^*$-coalgebra $\Bar_n(X\{n\})$.
\end{lemma}

\qed

\subsection{\sc The case $\mathcal{P}=[X,X]_\mathcal{Q}$}\label{sectionrel}

\subsubsection{\sc }
Recall some useful adjunction.

Let $\mathcal{Q}$ be an operad, $X,Y\in\Vect(\k)$  right $\mathcal{Q}$-modules, $S\in \Vect(\k)_\Sigma$. One has:
\begin{equation}\label{adjprima}
\Hom_{\Sigma, \Mod{-}\mathcal{Q}}(S\circ X,Y)=\Hom_{\Sigma}(S, [X,Y]_\mathcal{Q})
\end{equation}
See Section \ref{sectioncatalg} for the definition of $[X,Y]_\mathcal{Q}\in\Vect(\k)_\Sigma$.

\subsubsection{\sc }\label{sectionmaprel}
Let $\mathcal{Q}$ be an operad, $\mathcal{O}$ a Koszul operad, $X$ a symmetric sequence in $\Vect(\k)$, with a $\mathcal{O}{-}\mathcal{Q}$-bimodule structure:
\begin{equation}
m\colon \mathcal{O}\circ X\circ \mathcal{Q}\to X
\end{equation}
Recall that a $\mathcal{O}{-}\mathcal{Q}$-bimodule structure on a symmetric sequence $X$ gives rise to an operad map:
\begin{equation}
\mathcal{O}\to [X,X]_\mathcal{Q}
\end{equation}
See Section \ref{sectioncatalg} for detail.

\subsubsection{\sc }
Let $X$ be an $\mathcal{O}{-}\mathcal{Q}$-bimodule, as above.

Consider the bar-complex $$\Bar_\mathcal{O}(X)=(\mathcal{O}^\ash\circ X,d_\Bar)$$
The underlying symmetric sequence $\mathcal{O}^\ash\circ X$ is clearly a right $\mathcal{Q}$-module. 
It is shown in Section \ref{sectionbarrel} that $\Bar_{\mathcal{O}}(X)$ becomes an object of the category $\mathscr{C}(\mathcal{O}^\ash,\mathcal{Q})$, defined there.

The shifted symmetric sequence $X\{n\}$ becomes an $\mathcal{O}\{n\}{-}\mathcal{Q}\{n\}$-bimodule.

Consider the case $\mathcal{O}=\mathsf{e}_n$. 

Then the complex 
$$
\Bar_n(X)=(\mathsf{e}_n^*\circ (X\{n\}),d_\Bar)
$$
is an object of the category $\mathscr{C}(\mathsf{e}_n^*,\mathcal{Q}\{n\})$.

Let $f\colon X\to Y$ be a map of $\mathsf{e}_n{-}\mathcal{Q}$-bimodules, $X,Y$ symmetric sequences. We define 
a functor
\begin{equation}
_\mathcal{Q}F_{X,Y}^f\colon\Coalg\to\Sets
\end{equation}
as follows:
\begin{equation}\label{fqoriginal}
_\mathcal{Q}F_{X,Y}^f(a)=\big\{\phi\in\Hom_{\mathscr{C}(\mathsf{e}_n^*,\mathcal{Q}\{n\})}(a\star \Bar_n(X),\Bar_n(Y)),\ \phi\circ \eta=\Bar(f)\big\}
\end{equation}

We use Proposition \ref{baradjrel}. It gives:

\begin{equation}
_\mathcal{Q}F_{X,Y}^f(a)=\big\{\theta\in \MC(\Hom_{\Sigma, \Mod{-}\mathcal{Q}\{n\}}(a\star \Bar_n(X)/\k,Y\{n\})[-1]),\ \theta\circ \eta=\Bar(f)_\pr\big\}
\end{equation}

One further has:
\begin{equation}
\begin{aligned}
_\mathcal{Q}F_{X,Y}^f(a)=&\big\{\theta\in \MC(\Hom_\k(a,\Hom_{\Sigma,\Mod{-}\mathcal{Q}\{n\}}(\Bar_n(X)/\k,Y\{n\})[-1]),\ \theta\circ \eta=\Bar(f)_\pr\big\}\\
=&\big\{\theta\in \MC(\Hom_\k(a,\Hom_\Sigma(\mathsf{e}_n^*,[X\{n\},Y\{n\}]_{\mathcal{Q}\{n\}}))), \theta\circ\eta=\Bar(f)_\pr^\sim\big\}
\end{aligned}
\end{equation}
where by $\Bar(f)^\sim_\pr$ is denoted the element corresponded to $\Bar(f)_\pr$ by the adjunction \eqref{adjprima}. 

\subsubsection{\sc}\label{subsubsectionlierel}
Here we endow $\Hom_\Sigma(\mathsf{e}_n^*, [X\{n\},Y\{n\}]_{\mathcal{Q}\{n\}})[-1]$ with a dg Lie algebra structure.

Consider the symmetric sequence
\begin{equation}
D=\Hom_\lev(\mathsf{e}_n^*,[X\{n\},Y\{n\}]_{\mathcal{Q}\{n\}})
\end{equation}
The graded vector space $\Hom_\Sigma(\mathsf{e}_n^*, [X\{n\},Y\{n\}]_{\mathcal{Q}\{n\}})$ is equal to ${D}_\Sigma$.

On the other hand, the operad $[Y\{n\},Y\{n\}]_{\mathcal{Q}\{n\}}$ acts on $[X\{n\},Y\{n\}]_{\mathcal{Q}\{n\}}$ from the left:
\begin{equation}
[Y\{n\},Y\{n\}]_{\mathcal{Q}\{n\}}\circ [X\{n\},Y\{n\}]_{\mathcal{Q}\{n\}}\to [X\{n\},Y\{n\}]_{\mathcal{Q}\{n\}}
\end{equation}
see Section \ref{sectioncatalg}.

Therefore, the convolution operad
\begin{equation}
\Hom_\Op(\mathsf{e}_n^*,[Y\{n\},Y\{n\}]_{\mathcal{Q}\{n\}})
\end{equation}
acts on $D$, by Lemma \ref{lca}. Therefore, the same operad acts on the graded space ${D}_\Sigma=\Hom_\Sigma(\mathsf{e}_n^*,[X\{n\},Y\{n\}]_{\mathcal{Q}\{n\}})$.

There is a map of operads $\mathsf{e}_n\{n\}\to[Y\{n\},Y\{n\}]_{\mathcal{Q}\{n\}}$.
Therefore, $D$ and ${D}_\Sigma$ become algebras over the operad $\Hom_\Op(\mathsf{e}_n^*,\mathsf{e}_n\{n\})$.

Now we proceed as in Section \ref{defcomm12}.
There is a map of operads $\mathsf{Lie}\{1\}\to \Hom_\Op(\mathsf{e}_n^*,\mathsf{e}_n\{n\})$, by Lemma \ref{liekoszul}.
Therefore, the graded space
\begin{equation}\label{lierel}
\Hom_\Sigma(\mathsf{e}_n^*,[X\{n\},Y\{n\}]_{\mathcal{Q}\{n\}})[-1]
\end{equation}
is a graded Lie algebra. One easily checks that the differentials $d_0$ and $d_\Bar$ differentiate this Lie algebra structure.

Denote
\begin{equation}\label{def0q}
\Def_0(X,Y)_\mathcal{Q}:=(\Hom_\Sigma(\mathsf{e}_n^*,[X\{n\},Y\{n\}]_{\mathcal{Q}\{n\}})[-1], d=d_0+d_\Bar^\sim)
\end{equation}
We claim that the map $f\colon X\to Y$ of $\mathsf{e}_n{-}\mathcal{Q}$-bimodules defines a degree 1 element in $\Def_0(X,Y)_\mathcal{Q}$, satisfying the Maurer-Cartan equation. 

Indeed, $f$ defines a map of shifted bar-complexes
$$
\Bar(f)\{n\}\colon \Bar_n(X)\to\Bar_n(Y)
$$
which is a map of $\mathsf{e}_n^*$-coalgebras and of right $\mathcal{Q}\{n\}$-modules. The map $\Bar(f)\{n\}_\pr$ is defined by the projection to the cogenerators $\Bar(f)\{n\}_\pr\colon \Bar_n(X)\to Y\{n\}$. It is a map of right $\mathcal{Q}\{n\}$-modules. Now the adjunction \eqref{adjprima} gives the corresponding degree 0 map in $\Hom_{\Sigma}(\mathsf{e}_n^*,[X\{n\},Y\{n\}]_{\mathcal{Q}\{n\}})$.
After the shift $[-1]$ it gives a degree 1 element in $\Def_0(X,Y)_\mathcal{Q}$. One checks that it satisfies the Maurer-Cartan equation.

Denote the obtained element by $\Bar(f)\{n\}_\pr^\sim$, and
$$
d_f=\ad(\Bar(f)\{n\}_\pr^\sim)
$$
One gets that
\begin{equation}\label{defcomplexalg}
\Def(X\xrightarrow{f}Y)_\mathcal{Q}:=\big(\Hom_\Sigma(\mathsf{e}_n^*, [X\{n\},Y\{n\}]_{\mathcal{Q}\{n\}})[-1], d=d_0+d_\Bar+d_f\big)
\end{equation}
is a dg Lie algebra. We call it the {\it $\mathcal{Q}$-relative deformation dg Lie algebra}.

Note that
\begin{equation}
[X\{n\},Y\{n\}]_{\mathcal{Q}\{n\}}=([X,Y]_\mathcal{Q})\{n\}
\end{equation}
Therefore, one has:
\begin{equation}
\Def(X\xrightarrow{f}Y)_\mathcal{Q}=\big(\Hom_\Sigma(\mathsf{e}_n^\ash,[X,Y]_\mathcal{Q})[-1], d=d_0+d_\Bar+d_f\Big)
\end{equation}

\subsubsection{\sc }
We have:
\begin{prop}
The functor $_\mathcal{Q}F_{X,Y}^f\colon \Coalg(\mathsf{Comm})\to\Sets$ is representable, by the dg coalgebra 
$a_\rep=C_\CE(\Def(X\xrightarrow{f}Y)_\mathcal{Q},\k)$:
\begin{equation}
_\mathcal{Q}F_{X,Y}^f(a)=\Hom_{\Sigma,\Coalg(\mathsf{Comm})}(a, C_\CE(\Def(X\xrightarrow{f}Y)_\mathcal{Q},\k))
\end{equation}
where $\Coalg(\mathsf{Comm})$ denotes the category of pro-conilpotent cocommutative coalgebras in $\Vect(\k)_\Sigma$, and $C_\CE(-)$ stands for the chain Chevalley-Eilenberg complex of a (dg) Lie algebra.
\end{prop}
\begin{proof}
The proof is analogous to the proof of Proposition \ref{baradjrel}, and we omit the detail. The only difference is that we use Proposition \ref{baradjrel} here, instead of Proposition \ref{propmainconj}.
\end{proof}

\subsubsection{\sc}
Turn back to the definition of the functor $_\mathcal{Q}F_{X,Y}^f$, given in \eqref{fqoriginal}. It follows from this definition that, for a chain of maps of $\mathsf{e}_n{-}\mathcal{Q}$-bimodules in $\Vect(\k)_\Sigma$
$$
X\xrightarrow{f}Y\xrightarrow{g}Z
$$
one gets a map of sets
\begin{equation}\label{compq}
_\mathcal{Q}F_{Y,Z}^g(a^\prime)\times_\mathcal{Q}F_{X,Y}^f(a)\to _\mathcal{Q}F_{X,Z}^{gf}(a\otimes a^\prime)
\end{equation}
which gives rise to a map of bifunctors $\Coalg\times \Coalg\to\Sets$:
\begin{equation}
_\mathcal{Q}F^g_{Y,Z}(-_2)\times _\mathcal{Q}F^f_{X,Y}(-_1)\to _\mathcal{Q}F_{X,Z}^{gf}\circ \bigotimes(-_1,-_2)
\end{equation}
Denote by $a^f_\rep$ the representing coalgebra for the functor $_\mathcal{Q}F^f_{X,Y}(-)$. Then \eqref{compq} gives a map:
\begin{equation}\label{compq2}
a^f_\rep\otimes a^g_\rep\to a^{gf}_\rep
\end{equation}
of dg coalgebras,
which enjoys the natural associativity for a chain of four maps $X\to Y\to Z\to W$.

We know that 
$$
a_\rep^f=C_\CE(\Def(X\xrightarrow{f}Y)_\mathcal{Q},\k)
$$

Consider the case $X=Y$, $f=\id_X$. Then \eqref{compq2} gives a dg bialgebra structure on $$a_\rep^{\id}(X)=C_\CE(\Def(X\xrightarrow{\id}X)_\mathcal{Q},\k)$$
\begin{prop}
Let $X$ be a $\mathsf{e}_n{-}\mathcal{Q}$-bimodule. Then $$a_\rep^{\id}(X)=C_\CE(\Def(X\xrightarrow{\id}X)_\mathcal{Q},\k)$$ is a cocommutative dg bialgebra.
\end{prop}

One can apply the Milnor-Moore theorem to this cocommutative bialgebra, which gives that, as a graded bialgebra,
\begin{equation}
C_{\CE}(\Def(X\xrightarrow{\id}X)_\mathcal{Q},\k)=\mathscr{U}(\mathfrak{g})
\end{equation}
where $$\mathfrak{g}=\Def(X\xrightarrow{\id}X)_\mathcal{Q}[1]
$$
is the space of primitive elements. It becomes a Lie algebra, with the bracket $[x,y]=x*y\mp y*x$.

One has
\begin{lemma}\label{lemmazero2}
The Lie bracket on $\Def(X\xrightarrow{\id}X)_\mathcal{Q}$ defined in Section \ref{subsubsectionlierel} is equal to 0.
\end{lemma}
\begin{proof}
The proof is similar to Lemma \ref{lemmazero1}, and we omit the detail.
\end{proof}

\subsection{\sc The case of general $\mathcal{P}$}

\subsubsection{\sc }
Recall the general definition of the deformation complex $\Def_\Op(\mathcal{O}\xrightarrow{g}\mathcal{P})$ of a morphism of operads. 
To define $\Def_\Op(\mathcal{O}\xrightarrow{g}\mathcal{P})$, one replaces $\mathcal{O}$ by a cofibrant resolution $R(\mathcal{O})$, and one considers the operad maps $\mathrm{Map}_{\Op}(R(\mathcal{O}),\mathcal{P})$ in a formal neighborhood of the composition $R(\mathcal{O})\to\mathcal{O}\xrightarrow{g}\mathcal{P}$. In our case, when $\mathcal{O}$ is Koszul, one choses $R(\mathcal{O})=\Cobar_\Op(\mathcal{O}^\ash)$. 

We get:
\begin{equation}
\mathrm{Map}_{\Op}(R(\mathcal{O}),\mathcal{P})=\MC\Big(\Hom_\Sigma(\mathcal{O}^\ash,\mathcal{P}), d=d_0+D_g\Big)
\end{equation}
Here we consider the convolution Lie algebra $\Hom_\Sigma(\mathcal{O}^\ash,\mathcal{P})$, see Section \ref{section021}. The differential component $d_0$ comes from the inner differential on $\mathcal{P}$. The condition of compatibility with the cobar-differential is translated to the Maurer-Cartan equation, see [LV, Section 6.5]. The composition $\Cobar_\Op(\mathcal{O}^\ash)\to \mathcal{O}\xrightarrow{g}\mathcal{P}$ defines a Maurer-Cartan element $\omega_g$ in the convolution Lie algebra, and the differential component $D_g=\ad(\omega_g)$ is defined as the adjoint action of this element.

Finally we define the deformation complex of the operad morphism $g$ as
$$
\Def_\Op(\mathcal{O}\xrightarrow{g}\mathcal{P})=\Big(\Hom_\Sigma(\mathcal{O}^\ash,\mathcal{P}),d_0+D_g\Big)
$$

One has the following general statement:
\begin{lemma}\label{lemmaopdef}
Let $\mathcal{O}$ be a Koszul operad, $\mathcal{Q}$ an operad, $X$ an $\mathcal{O}{-}\mathcal{Q}$-bimodule.
Then the deformation complex $\Def(X\xrightarrow{\id}X)_\mathcal{Q}[1]$ is isomorphic to the deformation complex $\Def(\mathcal{O}\xrightarrow{f}[X,X]_\mathcal{Q})$:
\begin{equation}\label{goodiso}
\Def(X\xrightarrow{\id}X)_\mathcal{Q}[1]\simeq \Def(\mathcal{O}\xrightarrow{f}[X,X]_\mathcal{Q})
\end{equation}
 where the operad map $f\colon \mathcal{O}\to [X,X]_\mathcal{Q}$ is obtained from the bimodule structure on $X$, as in Section \ref{sectionmaprel}.
\end{lemma}
\begin{proof}
The isomorphism \eqref{goodiso} holds for the underlying graded vector spaces. Let us compare the differentials on both sides. The differential components $d_0$ are equal on the both sides. We claim that, up to the shift by [1],
\begin{equation}
D_g=d_\Bar+d_{\id}
\end{equation}
where we use the notations from \eqref{defcomplexalg}. Indeed, $D_g$ looks like the Hochschild cochain differential, it contains the regular and the two extreme terms, corresponded to $\omega_g\circ -$ and $-\circ \omega_g$, correspondingly. These two summands are equal to $d_\Bar$ and $d_{\id}$, correspondingly.

\end{proof}

One can alternatively describe the Lie bracket, given on $\Def(X\xrightarrow{\id}X)_\mathcal{Q}[1]$ by the Milnor-Moore theorem [Q, Appendix B], as follows. 
\begin{lemma}\label{lemmacompcomm}
The Lie bracket on $\Def(X\xrightarrow{\id}X)_\mathcal{Q}[1]$ defined via by the Milnor-Moore theorem is equal to the Lie bracket obtained from the convolution pre-Lie bracket on the operadic deformation complex $\Def_\Op(\mathsf{e}_n\xrightarrow{g}[X,X]_\mathcal{Q})$. For the case $\mathcal{Q}=I$, this Lie bracket can be also described as the Lie bracket on the graded space of coderivations of the cofree $\mathsf{e}_n^\ash$-coalgebra $\mathsf{e}_n^\ash\circ X$.
\end{lemma}
It is a direct check.

\qed

\subsubsection{\sc }
One particular case of previous construction is obtained when $X=\mathcal{Q}=\mathcal{P}$, with the tautological right $\mathcal{P}$-action on $X=\mathcal{P}$. We know that $[\mathcal{P},\mathcal{P}]_\mathcal{P}=\mathcal{P}$, see Section \ref{sectioncatalg}.
We assume we are given a map of operads $g\colon \mathsf{e}_n\to\mathcal{P}$, which gives, after the operadic shift, a map of operads $g\{n\}\colon\mathsf{e}_n\{n\}\to\mathcal{P}\{n\}$. Then $\mathcal{P}\{n\}$ becomes a $\mathsf{e}_n^{-}\mathcal{P}\{n\}$-bimodule. Lemma \ref{lemmabimod} gives a map of operads 
$$
\mathsf{e}_n\{n\}\to[\mathcal{P}\{n\},\mathcal{P}\{n\}]_{\mathcal{P}\{n\}}=\mathcal{P}\{n\}
$$
which is equal to the map $g\{n\}$.

Moreover, one easily shows that the deformation complexes $$\Def_\Op(\mathcal{O}\xrightarrow{g}\mathcal{P})\text{  and  }\Def_\Op(\mathcal{O}\{n\}\xrightarrow{g\{n\}}\mathcal{P}\{n\})$$ are isomorphic.

Therefore, the results of Section \ref{sectionrel} are applied to a general map of operads $g\colon \mathsf{e}_n\to\mathcal{P}$.

\section{\sc Deformation theory of a morphism of operads $f\colon \mathsf{e}_n\to\mathcal{P}$ with $\mathsf{e}_n$-coalgebra base}
\subsection{\sc The Hopf algebra structure on the operad $\mathsf{e}_n$}
\subsubsection{\sc }
Let $\mathcal{E}$ be a topological operad. Then the diagonal maps $\Delta(n)\colon \mathcal{E}(n)\to\mathcal{E}(n)\times\mathcal{E}(n)$, which gives rise to a map of operads $\Delta\colon \mathcal{E}\to\mathcal{E}\times\mathcal{E}$. That is, any topological operad is Hopf.

When we apply the monoidal functor of singular homology, we get an operad $H_\ldot(\mathcal{E},\k)$ in graded vector spaces, such that for any $n$ one has a map
$$
\Delta(n)\colon H_\ldot(\mathcal{E}(n),\k)\to H_\ldot(\mathcal{E}(n),\k)\otimes H_\ldot(\mathcal{E}(n),\k)
$$
It is compatible with the operad compositions, and gives rise to a map of operads
$$
\Delta\colon H_\ldot(\mathcal{E},\k)\to H_\ldot(\mathcal{E},\k)\otimes H_\ldot(\mathcal{E},\k)
$$
The conclusion is that the homology operad of a topological operad is always a Hopf operad (in graded vector spaces).
\subsubsection{\sc}\label{sectionhopfen}
It is applied to the operad $\mathsf{e}_n=H_\ldot(\mathsf{E}_n,\k)$, what makes it a Hopf operad. 

One can write down the Hopf structure on $\mathsf{e}_n$ explicitly. 
The suboperad $\mathsf{Comm}\subset \mathsf{e}_n$ is ``group-like'': the element $c_s\in \mathsf{Comm}(s)\subset\mathsf{e}_n(s)$, equal to the composition of the product operation $c_2\in\mathsf{Comm}(2)$, satisfies
\begin{equation}\label{hopfcomm}
\Delta(c_s)=c_s\otimes c_s
\end{equation}
The suboperad $\mathsf{Lie}\{-n+1\}\subset\mathsf{e}_n$ forms the subspace of primitive elements in each arity, in the sense that
\begin{equation}\label{hopflie}
\Delta(\ell_s)=\ell_s\otimes c_s+c_s\otimes \ell_s
\end{equation}
for any $\ell_s\in\mathsf{Lie}\{-n+1\}(s)$.

As $\mathsf{e}_n=\mathsf{Comm}\circ\mathsf{Lie}\{-n+1\}$, the condition that $\Delta$ is an operad map defines it on any component $\mathsf{e}_n(s)$. 

\subsubsection{\sc }
One has:

\begin{lemma}\label{2lie}
Consider the map of operads
\begin{equation}
\mathsf{Lie}\{1\}\to\mathsf{e}_n\{n\}\xrightarrow{\mathrm{Hopf}} \mathsf{e}_n\{n\}\otimes_\lev\mathsf{e}_n=\Hom_\lev(\mathsf{e}^*_n\{-n\},\mathsf{e}_n)=\Hom_\lev(\mathsf{e}_n^\ash,\mathsf{e}_n)
\end{equation}
where the leftmost map is the shifted map $\mathsf{Lie}\{-n+1\}\to\mathsf{e}_n$. Then the composition is equal to the map defined in Lemma \ref{liekoszul}.
\end{lemma}
\begin{proof}
We deal with two apriori different maps $\mathsf{Lie}\{1\}\to\mathcal{O}$ to an operad $\mathcal{O}$. The operad $\mathsf{Lie}$ is quadratic, therefore the two maps coincide if they coincide on $\mathsf{Lie}\{1\}(2)$. But for arity 2 it is checked by an explicit computation, using \eqref{hopfcomm} and \eqref{hopflie}, on one side, and \eqref{lieexplicit}, on the other side. 
\end{proof}

\subsubsection{\sc }
An immediate consequence of the Hopf structure on the operad $\mathsf{e}_n$ is:
\begin{lemma}\label{lemmatensorn}
Let $b_1,b_2$ be two $\mathsf{e}_n$-algebras. Then the tensor product of the underlying (dg) vector spaces $b_1\otimes b_2$ is naturally a $\mathsf{e}_n$-algebra. Similarly, let $a_1,a_2$ be two $\mathsf{e}_n^*$-coalgebras. Then the tensor product of the underlying (dg) vector spaces is naturally a $\mathsf{e}_n^*$-coalgebra.
\end{lemma}
\begin{proof}
We prove the first statement, the second is analogous. For $x$ an algebra over an operad $\mathcal{O}_1$ and $y$ an algebra over an operad $\mathcal{O}_2$, the tensor product $x\otimes y$ is an algebra over the operad $\mathcal{O}_1\otimes \mathcal{O}_2$.
In our case, $b_1\otimes b_2$ is an algebra over the operad $\mathsf{e}_n\otimes\mathsf{e}_n$. The diagonal map $\Delta\colon \mathsf{e}_n\to \mathsf{e}_n\otimes\mathsf{e}_n$ makes $b_1\otimes b_2$ an algebra over $\mathsf{e}_n$.
\end{proof}
We will use the following variation of the lemma:

\begin{lemma}\label{lemmatensornbis}
Let $b\in\Vect(\k), X\in\Vect(\k)_\Sigma$ be algebras over the operad $\mathsf{e}_n$. Then $b\star X=b^{(0)}\boxtimes X$ is an algebra over the operad $\mathsf{e}_n$. Similarly, for an $\mathsf{e}_n^*$-coalgebra $a\in\Vect(\k)$, and an $\mathsf{e}_n^*$-coalgebra $Y\in\Vect(\k)_\Sigma$, the symmetric sequence $a\star Y$ is a coalgebra over $\mathsf{e}_n^*$.
\end{lemma}
See Section \ref{defcomm11} for the definition of $a\star X$. The statement is proven analogously with the Lemma above, using Lemma \ref{lemmastarprod}.

\subsection{\sc The functor $_\mathcal{Q}G_{X,Y}^f$ and its representability}
\subsubsection{\sc }
Let $X$ be an $\mathsf{e}_n$-algebra in $\Vect(\k)_\Sigma$. Then the bar-complex $$\Bar_{n}(X)=(\mathsf{e}_n^*(X\{n\}),\ d_\Bar)$$ 
is an $\mathsf{e}_n^*$-coalgebra, see \eqref{barshifted1}, \eqref{barshifted2}. 

Let $a$ be another $\mathsf{e}_n^*$-coalgebra. Then $a\star\Bar_n(X)$ is a $\mathsf{e}_n$-coalgebra, by Lemma \ref{lemmatensornbis}.

We consider directly the ``relative'' case here. 

Let $\mathcal{Q}$ be an operad, $X,Y$-two $\mathsf{e}_n{-}\mathcal{Q}$-bimodules, $f\colon X\to Y$ a bimodule map.
The map $f$ defines a map $\Bar(f)\colon \Bar_n(X)\to \Bar_n(Y)$. It is a morphism in the category $\mathscr{C}(\mathsf{e}_n^*,\mathcal{Q}\{n\})$, see Section \ref{sectionbarrel}.

Denote by $\Coalg_n=\Coalg_n(\k)$ the category of pro-conilpotentent $\mathsf{e}_n^*$-coalgebras over $\k$.

Define the functor
$$
_\mathcal{Q}G_{X,Y}^f\colon\Coalg_n\to\Sets
$$
as follows:
\begin{equation}\label{gqoriginal}
_\mathcal{Q}G_{X,Y}^f(a)=\big\{\phi\in\Hom_{\mathscr{C}(\mathsf{e}_n^*,\mathcal{Q}\{n\})}(a\star \Bar_n(X),\Bar_n(Y)),\ \phi\circ \eta=\Bar(f)\big\}
\end{equation}
Here $\eta\colon \k\to a$ is the coaugmentation map.  

Proposition \ref{baradjrel} gives:
\begin{equation}\label{eqlongn}
\begin{aligned}
_\mathcal{Q}G^f_{X,Y}(a)&=\{\theta\in\MC(\Hom_{\Sigma,\Mod{-}\mathcal{Q}\{n\}}((a\star\Bar_n(X))/\k,
Y\{n\})[-1],\ \theta\circ \eta=\Bar(f)_\pr\}\\&=
\{\theta\in\MC(\Hom_\k(a,\Hom_{\Sigma,\Mod{-}\mathcal{Q}\{n\}}(\Bar_n(X)/\k,Y\{n\})[-1]),\ \theta\circ\eta=\Bar(f)_\pr\}\\&=
\{\theta\in\MC(\Hom_\k(a,\Hom_\Sigma(\mathsf{e}_n^*,[X\{n\},Y\{n\}]_{\mathcal{Q}\{n\}}))),\ \theta\circ\eta=\Bar(f)^\sim_\pr\}
\end{aligned}
\end{equation}
\subsubsection{\sc }
Denote
\begin{equation}
\Def_0^n(X\xrightarrow{f}Y)_\mathcal{Q}[1]=\Def_0(X\xrightarrow{f}Y)_\mathcal{Q}[1]=(\Hom_{\Sigma}(\mathsf{e}_n^*,[X\{n\},Y\{n\}]_{\mathcal{Q}\{n\}}), d=d_0+d_\Bar^\sim)
\end{equation}
cf. \eqref{def0q}.

\begin{lemma}\label{lemmadef0en}
In the notations as above, the complex $\Def_0^n(X\xrightarrow{f}Y)_\mathcal{Q}[1]$ enjoys a natural $\mathsf{e}_n\{n\}$-algebra structure. Moreover, the undelying $\mathsf{Lie}\{1\}$-algebra structure on $\Def_0^n(X\xrightarrow{f}Y)_\mathcal{Q}[1]$ agrees with the Lie algebra structure on $\Def_0(X\xrightarrow{f}Y)_\mathcal{Q}$ given in Section \ref{subsubsectionlierel}.
\end{lemma}

\begin{proof}
By Lemma \ref{lca}, the convolution operad $\Hom_\lev(\mathsf{e}_n^*,[Y\{n\},Y\{n\}]_{\mathcal{Q}\{n\}})$ acts on the symmetric sequence $\Hom_\lev(\mathsf{e}_n^*,[X\{n\},Y\{n\}]_{\mathcal{Q}\{n\}})$ and, therefore, on $\Hom_\Sigma(\mathsf{e}_n^*,[X\{n\},Y\{n\}]_{\mathcal{Q}\{n\}})$.
This action is compatible with the differential $d_0+d_\Bar^\sim$, which, in turn, acts on the components of the level Hom. 
Now we recall the operad map $\mathsf{e}_n\to [Y,Y]_\mathcal{Q}$, coming from the $\mathsf{e}_n{-}\mathcal{Q}$-bimodule structure, see Section \ref{sectionmaprel}. Then there is the shifted operad map $\mathsf{e}_n^*\{n\}\to [Y\{n\},Y\{n\}]_{\mathcal{Q}\{n\}}$, and the convolution operad $\Hom_\lev(\mathsf{e}_n^*,\mathsf{e}_n\{n\})=\mathsf{e}_n\otimes \mathsf{e}_n\{n\}$ becomes acting on 
$\Def_0^n(X\xrightarrow{f}Y)_\mathcal{Q}[1]$. Now the Hopf structure on the operad $\mathsf{e}_n$ (see Section \ref{sectionhopfen}) gives an operad map $\mathsf{e}_n\{n\}\to\mathsf{e}_n\otimes\mathsf{e}_n\{n\}$, which proves the first statement. 

For the second statement, we notice that the Lie algebra structure on $\Def_0(X\xrightarrow{f}Y)_\mathcal{Q}$ given in Section \ref{subsubsectionlierel} was constructed in the similar way, considering the action of the operad $\Hom_\lev(\mathsf{e}_n^*,[Y\{n\},Y\{n\}]_{\mathcal{Q}\{n\}})$ as the first step, followed by restricting this action to the action of operad $\Hom_\lev(\mathsf{e}_n^*,\mathsf{e}_n\{n\})$, and then restricting it to the operad $\mathsf{Lie}\{1\}$, by Lemma \ref{liekoszul}. So the claim is that the two $\mathsf{Lie}\{1\}$ actions coincide, and it follows from Lemma \ref{2lie}.

\end{proof}

\subsubsection{\sc }
We can twist the differential of deformation complex $\Def^n_0(X\xrightarrow{f}Y)_\mathcal{Q}$ (which, as a dg Lie algebra, does not depend on the map $f$) by the adjoint action $d_f=\ad(\Bar(f)\{n\}_\pr^\sim)$ of the Maurer-Cartan element $\Bar(f)\{n\}^\sim_\pr$, as in Section \ref{subsubsectionlierel}:
\begin{equation}\label{eqdefen}
\Def^n(X\xrightarrow{f}Y)_\mathcal{Q}:=\big(\Hom_\Sigma(\mathsf{e}_n^*, [X\{n\},Y\{n\}]_{\mathcal{Q}\{n\}})[-1], d=d_0+d_\Bar+d_f\big)
\end{equation}
The new thing is that this twisting preserves the operad $\mathsf{e}_n\{n\}$ acting on the shifted by [1] complex
$\Def^n(X\xrightarrow{f}Y)_\mathcal{Q}[1]$.

It can be checked by hand, of course. A more conceptual explanation goes through the second statement of Lemma \ref{lemmadef0en}. By this statement, the operad acting on the twisted complex is the operadic twisting $\Tw(\mathsf{Lie}\{1\}\to\mathsf{e}_n\{n\})$, see [DW]. This operad is well-known to be weak equivalent to the operad $\mathsf{e}_n\{n\}$, see loc.cit., Sect. 4.3.

We conclude, that $\Def^n(X\xrightarrow{f}Y)_\mathcal{Q}[1]$, defined in \eqref{eqdefen}, is a dg $\mathsf{e}_n\{n\}$-algebra.

\begin{remark}{\rm
Note that $\Def_0^n(X\xrightarrow{f}Y)_\mathcal{Q}$ is well-defined as a symmetric sequence, and only the last differential component  $d_f$ mixes the components of the symmetric sequence up. As well, the operadic twisting is a useful tool to be applied at this ``last step'', when the components of the symmetric sequence from the previous steps are being mixed up, and we are interesting which operad acts on the resulting complex. Many non-trivial examples of this idea in use can be found in [W], [CW], [DW], ...
}
\end{remark}

\subsubsection{\sc}
Consider the bar-complex $\Bar_{\mathsf{e}_n\{n\}}(\Def^n(X\xrightarrow{f}Y)_\mathcal{Q}[1])$ over the operad $\mathsf{e}_n\{n\}$. It is a $\mathsf{e}_n^*$-coalgebra. 

One has:
\begin{prop}\label{proprepn}
The functor $_\mathcal{Q}G^f_{X,Y}\colon \Coalg_n\to\Sets$ from the category of pro-conilpotent $\mathsf{e}_n^*$-coalgebras to the category of sets is representable, by the $\mathsf{e}_n^*$-coalgebra $A_\rep^f=\Bar_{\mathsf{e}_n\{n\}}(\Def^n(X\xrightarrow{f}Y)_\mathcal{Q}[1])$:
\begin{equation}
_\mathcal{Q}G_{X,Y}^f(a)=\Hom_{\mathscr{C}(\mathsf{e}_n^*,\mathcal{Q})}(a\star\Bar_n(X),\Bar_n(Y))=\Hom_{\Coalg_n}(a,\Bar_{\mathsf{e}_n\{n\}}(\Def^n(X\xrightarrow{f}Y)_\mathcal{Q}[1]))
\end{equation}
\end{prop}
\begin{proof}
We start with the last line of \eqref{eqlongn}, and then continue as in the proof of Proposition \ref{propreprbasic}, using the adjunction given in Proposition \ref{propbarcobarbis}. We skip the detail, as the argument is analogous to the proof of Proposition \ref{propreprbasic} (which is, by its own, borrowed from [T2, Prop. 3.2]).
\end{proof}

\subsection{\sc }
Turn back to the definition of the functor $_\mathcal{Q}G_{X,Y}^f$, given in \eqref{gqoriginal}. It follows from this definition that, for a chain of maps of $\mathsf{e}_n{-}\mathcal{Q}$-bimodules in $\Vect(\k)_\Sigma$
$$
X\xrightarrow{f}Y\xrightarrow{g}Z
$$
one gets a map of sets
\begin{equation}\label{compqn}
_\mathcal{Q}G_{Y,Z}^g(a^\prime)\times_\mathcal{Q}G_{X,Y}^f(a)\to _\mathcal{Q}G_{X,Z}^{gf}(a\otimes a^\prime)
\end{equation}
which gives rise to a map of bifunctors $\Coalg_n\times \Coalg_n\to\Sets$:
\begin{equation}
_\mathcal{Q}G^g_{Y,Z}(-_2)\times _\mathcal{Q}G^f_{X,Y}(-_1)\to _\mathcal{Q}G_{X,Z}^{gf}\circ \bigotimes(-_1,-_2)
\end{equation}
Denote by $A^f_\rep$ the representing $n$-coalgebra for the functor $_\mathcal{Q}G^f_{X,Y}(-)$. Then \eqref{compq} gives a map:
\begin{equation}\label{compq2n}
A^f_\rep\otimes A^g_\rep\to A^{gf}_\rep
\end{equation}
of dg coalgebras,
which enjoys the natural associativity for a chain of four maps $X\to Y\to Z\to W$.

We know from Proposition \ref{proprepn} that
$$
A_\rep^f=\Bar_{\mathsf{e}_n\{n\}}(\Def^n(X\xrightarrow{f}Y)_\mathcal{Q}[1])
$$

Consider the case $X=Y$, $f=\id_X$. Then \eqref{compq2n} gives a monoid in the category of dg $n$-coalgebras structure on $A_\rep^{\id}(X)=\Bar_{\mathsf{e}_n\{n\}}(\Def^n(X\xrightarrow{\id}X)_\mathcal{Q})$:
\begin{prop}\label{propshort}
Let $X$ be a $\mathsf{e}_n{-}\mathcal{Q}$-bimodule. Then $$A_\rep^{\id}(X)=\Bar_{\mathsf{e}_n\{n\}}(\Def^n(X\xrightarrow{\id}X)_\mathcal{Q}[1])$$ is a monoid in the category of $\mathsf{e}_n^*$-coalgebras.
\end{prop}

\begin{defn}{\rm
Let $A$ be a $\mathsf{e}_n^*$-coalgebra, endowed with a monoid structure, that is, with a map 
\begin{equation}\label{prodbialgn}
A\otimes A\to A
\end{equation}
of $\mathsf{e}_n^*$-coalgebras, which is associative, and has a unit given by the coaugmentation map of $A$. Then we say that $A$ is {\it a $\mathsf{e}_n^*$-bialgebra}. 
}
\end{defn}

Proposition \ref{proprepn} says that the $\mathsf{e}_n^*$-coalgebra $$A_\rep^{\id}(X)=\Bar_{\mathsf{e}_n\{n\}}(\Def^n(X\xrightarrow{\id}X)_\mathcal{Q}[1])$$
is a $\mathsf{e}_n^*$-bialgebra.

We know from Lemma \ref{lemmaopdef} that, as a complex, $\Def^n(X\xrightarrow{\id}X)_\mathcal{Q}[1]=\Def_{\Op}(\mathsf{e}_n\to [X,X]_\mathcal{Q})$. 

Then Proposition \ref{propshort} gives:
\begin{coroll}\label{corollfinal0}
Let $X$ be a $\mathsf{e}_n\{n\}{-}\mathcal{Q}$-bimodule, $g\colon \mathsf{e}_n\{n\}\to [X,X]_{\mathcal{Q}}$ the associated map. Then the operadic deformation complex $\Def_\Op(\mathsf{e}_n\{n\}\xrightarrow{g}[X,X]_\mathcal{Q})$ is a $\mathsf{e}_n\{n\}$-algebra, and $\Bar_{\mathsf{e}_n\{n\}}(\Def_\Op(\mathsf{e}_n\{n\}\xrightarrow{g}[X,X]_\mathcal{Q}))$ has a $\mathsf{e}_n^*$-bialgebra structure.
\end{coroll}

One particular case of previous construction is obtained when $X=\mathcal{Q}=\mathcal{P}$, with the tautological right $\mathcal{P}$-action on $X=\mathcal{P}$. We know that $[\mathcal{P},\mathcal{P}]_\mathcal{P}=\mathcal{P}$, see Section \ref{sectioncatalg}.
We assume we are given a map of operads $f\colon \mathsf{e}_n\to\mathcal{P}$, then $f\{n\}\colon \mathsf{e}_n\{n\}\to\mathcal{P}\{n\}$, which makes $\mathcal{P}\{n\}$ a $\mathsf{e}_n\{n\}{-}\mathcal{P}\{n\}$-bimodule. Lemma \ref{lemmabimod} gives a map of operads 
$$
\mathsf{e}_n\{n\}\to[\mathcal{P}\{n\},\mathcal{P}\{n\}]_{\mathcal{P}\{n\}}=\mathcal{P}\{n\}
$$
which is equal to the map $f\{n\}$. We mention that $\Def_\Op(\mathcal{O}\xrightarrow{f}\mathcal{P})=\Def_\Op(\mathcal{O}\{k\}\xrightarrow{f\{k\}}\mathcal{P}\{k\})$, for any $k$.

One gets:
\begin{coroll}\label{corollfinal}
Let $f\colon \mathsf{e}_n\to \mathcal{P}$ be an operad map. Then $\Def_\Op(\mathsf{e}_n\xrightarrow{f}\mathcal{P})$ is a $\mathsf{e}_n\{n\}$-algebra, and $\Bar_{\mathsf{e}_n\{n\}}(\Def_\Op(\mathsf{e}_n\xrightarrow{f}\mathcal{P}))$ is a $\mathsf{e}_n^*$-bialgebra.
\end{coroll}

In the next Section, we link the dg cocommutative bialgebra ${a_\rep^{\id}}(X)$ and the dg $\mathsf{e}_n^*$-bialgebra ${A_\rep^{\id}}(X)$, and investigate the higher structure we have found.

\section{\sc A proof of Theorem \ref{theorem1}}
Section \ref{section61} just reproduces arguments from [T2]. For convenience of the reader, we briefly recall them here.
It shows that $\Def_\Op(\mathsf{e}_n\{n\}\xrightarrow{f\{n\}}\mathcal{P}\{n\})[-n]$ is a homotopy $(n+1)$-algebra. 
Note that here we only use the deformation theory with $\mathsf{e}_n$-coalgebra base, see Section 5.

In Section \ref{section62}, we give an explicit description of the underlying (homotopy) Lie bracket on $\Def_\Op(\mathsf{e}_n\{n\}\xrightarrow{f\{n\}}\mathcal{P}\{n\})[-n][n]=\Def_\Op(\mathsf{e}_n\{n\}\xrightarrow{f\{n\}}\mathcal{P}\{n\})$ of this homotopy $\mathsf{e}_{n+1}$-algebra as the strict operadic convolution Lie bracket. It is proven by comparison of the representation objects $a_\rep(X)$ and $A_\rep(X)$ for the deformation theories with cocommutative and $\mathsf{e}_n$-coalgebra bases, correspondingly. Our argument here is hopefully somewhat more transparent than the one in [T2].
\subsection{\sc }\label{section61}
Let $f\colon \mathsf{e}_n\to\mathcal{P}$ be a morphism of operads; it defines the shifted map of operads $f\{n\}\colon\mathsf{e}_n\{n\}\to\mathcal{P}\{n\}$. The deformation complex $\Def_\Op(\mathsf{e}_n\xrightarrow{f}\mathcal{P})=\Def_{\Op}(\mathsf{e}_n\{n\}\xrightarrow{f\{n\}}\mathcal{P}\{n\})$ enjoys a structure of $\mathsf{e}_n\{n\}$-algebra, by Lemma \ref{lemmadef0en} and the discussion thereafter.

Consider the bar-complex $\Bar_{\mathsf{e}_n\{n\}}(\Def_\Op(\mathsf{e}_n\{n\}\xrightarrow{f\{n\}}\mathcal{P}\{n\}))$ which is an $\mathsf{e}_n^*$-coalgebra.

In general, the underlying complex of $\Bar_{\mathsf{e}_n\{n\}}(X)$ is $\mathsf{e}_n^*\circ X$. As $\mathsf{e}_n=\mathsf{Comm}\circ \mathsf{Lie}\{-n+1\}$, the underlying graded space of 
\begin{equation}
\mathsf{e}_n^*\circ X=\mathsf{Comm}^*\circ (\mathsf{Lie}^*\{n-1\}\circ X)
\end{equation}

Denote $X=\Def_\Op(\mathsf{e}_n\{n\}\xrightarrow{f\{n\}}\mathcal{P}\{n\})$.

Corollary \ref{corollfinal} says that $\Bar_{\mathsf{e}_n\{n\}}(X)$ is a $\mathsf{e}_n^*$-bialgebra. Consider the underlying cocommutative bialgebra. Then the Milnor-Moore theorem [Q, Appendix B] gives a Lie algebra structure on the graded space of primitive elements.
The space of primitive elements is $\mathsf{Lie}^*\{n-1\}\circ X$. In general, if we have a dg cocommutative coalgebra, the space of primitive elements is a subcomplex. The space of primitive elements is closed under the bracket $ab\mp ba$; therefore, $\mathsf{Lie}^*\{n-1\}\circ X$ becomes a Lie algebra. The Lie algebra structure on $\mathsf{Lie}^*\{n-1\}\circ X$ is compatible with the cofree  $\mathsf{Lie}^*\{n-1\}$-coalgebra structure, by 
\begin{equation}\label{lien}
\delta([x,y])=[\delta x,y]+(-1)^{|x|(d-1)}[x,\delta y]
\end{equation}
where $\delta$ is the $\mathsf{Lie}^*\{n-1\}$-cobracket.

The Chevalley-Eilenberg chain complex of the Lie algebra $\mathsf{Lie}^*\{n-1\}\circ X$ is 
$
\mathsf{Comm}^*\{-1\}\circ \mathsf{Lie}^*\{n-1\}\circ X
$
The identity \eqref{lien} is translated to the statement that the Chevalley-Eilenberg differential differentiates the $\mathsf{Lie}^*\{n-1\}$-coalgebra structure, making $\mathsf{Comm}^*\{-1\}\circ \mathsf{Lie}^*\{n-1\}\circ X$ a dg $\mathsf{Comm}^*\{-1\}\circ \mathsf{Lie}^*\{n-1\}$-coalgebra.

One has:
$$
\mathsf{Comm}^*\{-1\}\circ \mathsf{Lie}^*\{n-1\}\circ X=(\mathsf{Comm}^*\{-n-1\}\circ \mathsf{Lie}^*\{-1\}\circ X\{-n\})\{n\}=
(\mathsf{e}_{n+1}^*\{-n-1\}\circ X[-n])[n]
$$
As $(\mathsf{Comm}^*\{-1\}\circ \mathsf{Lie}^*\{n-1\}\circ X,\ d)$ is a dg coalgebra over the cooperad $\mathsf{Comm}^*\{-1\}\circ \mathsf{Lie}^*\{n-1\}$, the shifted by $[-n]$ complex $\mathsf{Comm}^*\{-n-1\}\circ \mathsf{Lie}^*\{-1\}\circ X\{-n\}$ is a coalgebra over the cooperad $\Big(\mathsf{Comm}^*\{-1\}\circ \mathsf{Lie}^*\{n-1\}\Big)\{-n\}=\mathsf{e}_{n+1}^*\{-n-1\}$.

Here have we repeatedly used Lemma \ref{lemmashift}(i).

The conclusion is that $\mathsf{e}_{n+1}^*\{-n-1\}\circ (X[-n])$ is a dg coalgebra over  $\mathsf{e}_{n+1}^*\{-n-1\}$.
By definition, it means that $X[-n]$ is a homotopy $(n+1)$-algebra.

\subsection{\sc }\label{section62}
It remains to prove the second claim of Theorem \ref{theorem1}, describing the Lie bracket of degree $-n$ in terms of the operadic convolution bracket. 
\subsubsection{\sc }
The idea is to link the commutative coalgebra $a_\rep(f)$ representing the functor $_\mathcal{Q}F_{X,Y}^f$ with the $n$-coalgebra $A_\rep(f)$ representing the functor $_\mathcal{Q}G_{X,Y}^f$. 

There is the inclusion functor $i\colon\Coalg\to\Coalg_n$. It admits a right adjoint $R\colon \Coalg_n\to\Coalg$, defined as follows. 
The cocommutative coalgebra $R(a)$ is defined as the biggest cocommutative coalgebra contained in $\Ker\delta$, where $\delta$ is the cobracket (see [T2, Prop. 4.4]). 

Therefore, for any $a\in \Coalg$, one has:
\begin{equation}
\Hom_{\Coalg_n}(i(a), A_\rep(f))=\Hom_{\Coalg}(a, R(A_\rep(f)))
\end{equation}
On the other hand, 
\begin{equation}
\Hom_{\Coalg_n}(i(a),A_\rep(f))=_\mathcal{Q}G_{X,Y}^f(i(a))=_\mathcal{Q}F_{X,Y}^f(a)=\Hom_{\Coalg}(a, a_\rep(f))
\end{equation}
It follows from the Yoneda lemma that 
\begin{equation}\label{eqconclusion}
R(A_\rep(f))=a_\rep(f)
\end{equation}
The functor $R$ is lax-monoidal (it follows either from an explicit computation, or from the general fact that a right adjoint to a colax-monoidal functor is lax-monoidal; the functor $i$ is strict monoidal and in particular colax-monoidal).
Moreover, \eqref{eqconclusion} is compatible with the composition properties
$$
a_\rep(f)\otimes a_\rep(g)\to a_\rep(gf)
$$
and
$$
A_\rep(f)\otimes A_\rep(g)\to A_\rep(gf)
$$
for a chain of morphisms $X\xrightarrow{f}Y\xrightarrow{g}Z$, in the sense that the diagram
\begin{equation}
\xymatrix{
R(A_\rep(f)\otimes A_\rep(g))\ar[r]& R(A_\rep(gf))\ar[r]^{=}&a_\rep(gf)\\
R(A_\rep(f))\otimes R(A_\rep(g))\ar[u]\ar[rr]^{=\otimes =}&&a_\rep(f)\otimes a_\rep(g)\ar[u]
}
\end{equation}
commutes, where the  left vertical arrow is the lax monoidal map. It is straightforward.

Thus we get:
\begin{prop}\label{propthelast}
Let $X$ be a $\mathsf{e}_n{-}\mathcal{Q}$-bimodule. Denote by $a_\rep(\id_X)$ (corresp., $A_\rep(\id_X)$)  the commutative coalgebra (corresp., the $n$-coalgebra), representing the functor $_\mathcal{Q}F_{X,X}^{\id}$ (corresp., $_\mathcal{Q}G_{X,X}^{\id}$). Then the functor $R$ defined above yields a map $R\colon A_\rep(\id_X)\to a_\rep(\id_X)$ which is a map of monoid objects. 
\end{prop}

\subsubsection{\sc }
In the case of a cofree $n$-coalgebra $\mathsf{e}_n^*\circ V$, the result of application of the functor $R$ is the cofree cocommutative coalgebra cogenerated by $V$:
\begin{equation}
R(\mathsf{e}_n^*\circ V)=\mathsf{Comm}^*\circ V
\end{equation}
In our case
\begin{equation}
A_\rep(f)=\Bar_{\mathsf{e}_n\{n\}}(\Def(X\xrightarrow{f}Y)_\mathcal{Q})=\Big(\mathsf{e}_n^*\circ (\Def(X\xrightarrow{f}Y)_\mathcal{Q}), d\Big)
\end{equation}
and
\begin{equation}
a_\rep(f)=\Bar_{\mathsf{Lie}\{1\}}(\Def(X\xrightarrow{f}Y)_\mathcal{Q})=\Big(\mathsf{Comm}^*\circ (\Def(X\xrightarrow{f}Y)_\mathcal{Q}[1])), d^\prime\Big)
\end{equation}

We have a priori two different  Lie algebra structures on $\Def(X\xrightarrow{\id}X)_\mathcal{Q}[1]$, defined from the monoid structures on $a_\rep(\id_X)$ and on $A_\rep(\id_X)$. Proposition \ref{propthelast} and the computation thereafter imply that 
these two Lie algebra structures are equal. On the other hand, we know the one obtained on $a_\rep(\id_X)$, from Lemma \ref{lemmacompcomm}. It completes the proof of Theorem \ref{theorem1}.

\comment
\appendix
\section{\sc Appendix}

\subsection{\sc Proof of Proposition \ref{lemmabm}}

The proof is relied on the following lemmas:
\begin{lemma}\label{lemmaprop1}
Let $G$ be a group, $H\subset G$ a subgroup. Let $V,W$ be $H$-modules.
Then the following is true:
\begin{itemize}
\item[(i)] there is a functorial map
\begin{equation}
\Ind_H^G(V\otimes W)\to \Ind_H^G(V)\otimes \Ind_H^G(W)
\end{equation}
\item[(ii)] there is a functorial map
\begin{equation}
\Coind_H^G(V)\otimes \Coind_H^G(W)\to\Coind_H^G(V\otimes W)
\end{equation}
\item[(iii)] there is a functorial map
\begin{equation}
\Hom(\Ind_H^GV,\Coind_H^GW)\to \Coind_H^G\Hom(V,W)
\end{equation}
\item[(iv)]
\begin{equation}
\Ind_H^G\Hom(V,W)\to\Hom(\Coind_H^GV,\Ind_H^GW)
\end{equation}
\end{itemize}
\end{lemma}

\begin{proof}
The first two statements follow from the adjunctions 
\begin{equation}
\Hom_G(\Ind_H^GV,W)=\Hom_H(V,\Res W)
\end{equation}
and
\begin{equation}
\Hom_G(W,\Coind_H^GV)=\Hom_H(\Res W, V)
\end{equation}
correspondingly.

For the third statement, we note that it follows from (ii) for the case when $V$ is finite-dimensional.
For general $V$, one has $V=\colim V_i$ (the union), where each $V_i$ is finite-dimensional.
Then we have:
\begin{equation}
\begin{aligned}
\ &\Hom(\Ind_H^G(\colim V_i),\Coind_H^GW)=\Hom(\colim\  \Ind_H^G V_i,\Coind_H^GW)=\lim\Hom(\Ind_H^GV_i,\Coind_H^GW)\to\\
&\lim \Coind_H^G\Hom(V_i,W)=\Coind_H^G\lim\Hom(V_i,W)=\Coind_H^G\Hom(\colim\  V_i,W)=\Coind_H^G\Hom(V,W)
\end{aligned}
\end{equation}
Here in the first equality we use that $\Ind_H^G$ is left adjoint and therefore it commutes with colimits, and in the first equality of the second line we use that $\Coind_H^G$ is right adjoint and therefore commutes with the limits.

The proof of (iv) is analogous to (iii). We take to the account that, for $H$ and $G$ finite, one has $\Ind_H^G=\Coind_H^G$. Therefore, either of these two functors commutes with both limits and colimits, what gives us an extra flexibility.

\end{proof}

\begin{lemma}\label{lemmaprop2}
Let $G$ be a group, and $A,B$ be $G$-modules. Then the following statements are true:
\begin{itemize}
\item[(i)] there is a lax-monoidal map of invariants:
\begin{equation}
A^G\otimes B^G\to (A\otimes B)^G
\end{equation}
\item[(ii)] there is a colax-monoidal map of coinvariants
\begin{equation}
(A\otimes B)_G\to A_G\otimes B_G
\end{equation}
\item[(iii)] there is a map
\begin{equation}
\Hom(A_G,B^G)\to \Hom(A,B)^G
\end{equation}
\item[(iv)] 
there is a map
\begin{equation}
 \Hom(A,B)_G\to \Hom(A^G,B_G)
\end{equation}
\end{itemize}
\end{lemma}
\begin{proof}
The statements (i) and (ii) are clear. For (iii) and (iv), note that in the assumption $\cchar\ k=0$, $V_G\simeq V^G$ for any $V$. Then either of these functors commutes with both limits and colimits. Then we continue as in the proof of Lemma \ref{lemmaprop1}(iii).
\end{proof}

Turn back to the proof of Proposition. 

Let $f\in \Hom(X(k),Y(k))$, $f_i\in\Hom(X_1(n_i),Y_1(n_i))$, $i=1,\dots,k$.
Take $G=\Sigma_{n_1+\dots+n_k}$, $H=\Sigma_{n_1}\times\dots\times\Sigma_{n_k}$. Lemma \ref{lemmaprop1} gives
\begin{equation}
\begin{aligned}
\ &\Ind_{H}^G\Big(\bigotimes_{i=1}^k\Hom(X_1(n_i),Y_1(n_i))\Big)\to
\bigotimes_{i=1}^k\Ind_H^G\Hom(X_1(n_i),Y_1(n_i))\to\\
&\bigotimes_{i=1}^k\Hom(\Coind_H^GX_1(n_i),\Ind_H^GY_1(n_i))\to
\Hom\Big(\bigotimes\Coind_H^GX_1(n_i),\bigotimes\Ind_H^GY_1(n_i)\Big)=\\
&\Hom\Big(\bigotimes\Ind_H^GX_1(n_i),\bigotimes\Coind_H^GY_1(n_i)\Big)\to\Hom\Big(\Ind_H^G\big(\otimes_{i=1}^kX_1(n_i)\big),\Coind_H^G\big(\otimes_{i=1}^kY_1(n_i)\big)\Big)=\\
&\Hom\Big(\Coind_H^G\big(\otimes_{i=1}^kX_1(n_i)\big),\Ind_H^G\big(\otimes_{i=1}^kY_1(n_i)\big)\Big)
\end{aligned}
\end{equation}
Both sides are representations of the group $G$, they enjoy as well a commuting with $G$-action an action of the symmetric group $G_1=\Sigma_k$. 
Applying Lemma \ref{lemmaprop2} to $G=G_1$, we get the statement of Proposition.

\endcomment

\bigskip
{\small
\noindent {\sc Universiteit Antwerpen, Campus Middelheim, Wiskunde en Informatica, Gebouw G\\
Middelheimlaan 1, 2020 Antwerpen, Belgi\"{e}}}

\vspace{1mm}

{\small
\noindent{\sc Laboratory of Algebraic Geometry,
National Research University Higher School of Economics,
Moscow, Russia}}

\bigskip

\noindent{{\it e-mail}: {\tt Boris.Shoikhet@uantwerpen.be}}
\end{document}